\documentclass[oneside, 11pt]{amsart}

\title[] {Newton-Okounkov bodies and normal toric degenerations}
\author{Georg Merz} 

\usepackage{xcolor}
\usepackage{booktabs}
\usepackage{enumerate}
\usepackage[ruled,vlined]{algorithm2e}
\usepackage{amsmath}
\usepackage{mathrsfs}
\usepackage{amssymb}
\usepackage{amsthm}
\usepackage[all]{xy}
\usepackage{tikz}
\usepackage{mathtools}  
\mathtoolsset{showonlyrefs}
\usepackage{caption}
\usepackage{longtable}
\usepackage{makecell}
\usepackage{tabularx}
\usepackage{float}
\restylefloat{table}

\keywords{Newton-Okounkov body, toric degeneration, normal semigroups, del Pezzo surface}
\date{\today}

\address{Georg Merz,
Mathematisches Institut,
Georg-August-Universit\"at G\"ottingen,
Bunsenstra\ss e 3-5, 
D-37073 G\"ottingen,
Germany}
\email{georg.merz@mathematik.uni-goettingen.de}

\DeclareMathOperator{\proj}{Proj}

\newcommand{\struc}{\mathcal{O}}

  \theoremstyle{plain}
  \newtheorem{thm}{Theorem}[section]
  \newtheorem*{thmnonr}{Theorem}
  \newtheorem{remark}[thm]{Remark}
  \newtheorem*{quest}{Question}
  \newtheorem*{answ}{Answer}
  \newtheorem{lem}[thm]{Lemma}
  \newtheorem{cor}[thm]{Corollary}
  \newtheorem{prop}[thm]{Proposition}
  \theoremstyle{definition}
  \newtheorem{defi}[thm]{Definition}
  \newtheorem{ex}[thm] {Example}
\theoremstyle{theorem}

\usepackage{bookmark}
\begin{document}
\maketitle
\begin{abstract}
Anderson proved that  the finite generation of the value semigroup $\Gamma_{Y_\bullet}(D)$ in the construction of the Newton-Okounkov body $\Delta_{Y_\bullet}(D)$ induces a toric degeneration of the corresponding variety $X$ to some toric variety $X_0$. In this case the normalization of $X_0$ is the normal toric variety corresponding to the rational polytope $\Delta_{Y_\bullet}(D)$. Since $X_0$ is not normal in general this correspondence is rather implicit. In this article we investigate in conditions to assure that $X_0$ is normal, by comparing the Hilbert polynomial with the Ehrhart polynomial.
 In the case of del Pezzo surfaces this will result in an algorithm which outputs for a given divisor $D$ a flag $Y_\bullet$ such that the value semigroup in question is indeed normal. Furthermore, we will find flags on del Pezzo surfaces and on some particular weak del Pezzo surfaces which induce normal toric degenerations for all possible divisors at once. We will prove that in this  case  the global value semigroup $\Gamma_{Y_\bullet}(X)$ is finitely generated and normal.
\end{abstract}
\section{introduction}
Newton-Okounkov bodies are convex bodies associated to linear series on a projective variety. They were introduced by Okounkov \cite{O96} and further systematically studied by Lazarsfeld-Musta{\c{t}}{\u{a}} \cite{LM09} and Kaveh-Khovanskii \cite{KK12}. Newton-Okounkov bodies of a linear series $|V|$ are not unique but depend upon the choice of a valuation on the graded algebra of sections  $R(|V|)$. In the special case of $X$ being toric and $D$ a torus invariant divisor, one can define a valuation such that the associated Newton-Okounkov body is, up to translation, the polytope $\Delta(D)$ corresponding to $D$ in the sense of the usual toric correspondence (see \cite[Proposition 6.1]{LM09}). In general, for an arbitrary projective variety $X$ and a valuation $\nu$ the Newton-Okounkov body does not need to be rational polyhedral. However, if $\Delta_{\nu}(D)$ is rational polyhedral, one can ask the following question.
\begin{quest}
Assuming $\Delta_{\nu}(D)$ is rational polyhedral. What is the connection between $X$ and the toric variety corresponding to $\Delta_{\nu}(D)$?
\end{quest}
The answer to this question was given by D. Anderson. He showed the following.
\begin{thmnonr}[\cite{A13}]
Let $X$ be a projective variety, $D$ a very ample (Cartier) divisor, and $\nu$ a valuation-like function. Assume that the semigroup $\Gamma=\Gamma_{\nu}(D)=\{(\nu(s),k) \ | \ s\in H^0(X,\struc_{X}(kD)), \ k\in\mathbb{N}\}$ is finitely generated. Then there exists a toric degeneration of $X$ w.r.t. $D$  to the toric  variety $X_0:= \proj (k[\Gamma])$. Moreover, the normalization of $X_0$ is the normal toric variety corresponding to the polytope $\Delta_{\nu}(D)$.
\end{thmnonr}
Anderson's Theorem can bee seen as a generalization of the theory of SAGBI bases (see \cite[Chapter 11]{S96} for an introduction).
In the SAGBI case, one of the prerequisites is that the coordinate ring $k[X]$ of the corresponding variety $X$ needs to be contained in a polynomial ring $k[T_1,\dots, T_n]$. But this is quite a strong constraint, which we can  omit by considering the valuation $\nu$.

However, the connection between the variety $X$ and the normal toric variety corresponding to the polytope $\Delta_{\nu}(D)$ is rather implicit, since we need to normalize the variety we degenerate to. Hence, we can raise the following question.

\begin{quest}
Under which circumstances, does there exist a degeneration of $X$ to the normal toric variety corresponding to $\Delta_{\nu}(D)$?
\end{quest} 
In order to answer this question, we need to determine when the variety $\proj(k[\Gamma])$ is normal.
This is the case if and only if there is a $k\in\mathbb{N}$ such that the semigroup $k\cdot \Gamma$ is normal, i.e. $\text{Cone}(k\cdot \Gamma)\cap \mathbb{Z}^d=k\cdot \Gamma$ (see also Section \ref{sectionnormalaffine} for more details). 

We will see that the property of inducing a normal toric degeneration can indeed be checked by considering the shape of $\Delta_{Y_\bullet}(D)$, or more concretely the Ehrhart polynomial of $\Delta_{Y_\bullet}(D)$. We will define the difference between the Ehrhart polynomial of the Newton-Okounkov body and the hilbert polynomial of $D$ as the \emph{normal defect}. It is then not difficult to prove that this difference is zero if and only if $\Delta_{Y_\bullet}(D)$ induces a normal toric degeneration. This gives the following answer.
\begin{answ}
$\Delta:=\Delta_{\nu}(D)$ induces a normal toric degeneration if the Ehrhart polynomial corresponding to $\Delta$ is equal to the Hilbert polynomial corresponding to $D$.
\end{answ}

 This observation enables us to view the problem of finding a flag for a given divisor which induces a normal toric degeneration as an optimization problem. 
We will evolve this idea further in the case where $X$ is a surface. It turns out that one can formulate this optimization problem in the following form:

{\it Given a divisor $D$, find a flag $Y_\bullet$ such that the number of integral points on the boundary of $\Delta_{Y_\bullet}(D)$ is minimal.}

We will indeed prove that under some condition (e.g. if $X$ is a Mori dream surface)  such a flag always exists (see  Theorem \ref{thmexminsurf}).
If we additionally assume that the Zariski decomposition of $X$ is integral (e.g. for del Pezzo surfaces), we will give a concrete algorithm in order to find such flags. 

Finally, we focus on (weak) del Pezzo surfaces. It will turn out that in this situation, negative curves are good candidates for flags inducing normal toric degeneration. More concretely, we prove the following statement.

\begin{thmnonr}
Suppose one of the following situations.
\begin{itemize}
\item $X=X_r$ is the blow-up of $1\leq r\leq 6$ points in general position and $Y_\bullet$ is an admissible flag such that $Y_1$ is negative. 
\item $X=L_3$ is the blow-up of four points, where three of them are on a line or $X=S_6$ is the blow-up of six points on a conic. Let $Y_\bullet$ be an admissible flag such that $Y_1$ is the unique $(-2)$-curve on $X$.
\end{itemize}
Then the global semigroup 
\begin{align}
\Gamma_{Y_\bullet}(X)=\{(\nu_{Y_\bullet}(s),D) \ | \ D\in \text{Pic}(X)=N^1(X), \ s\in H^0(X,\mathcal{O}(D)) \}
\end{align}
 is finitely generated normal.
\end{thmnonr}
In order to proof such a statement, one first needs to prove the finite generation and normality of the value semigroup $\Gamma_{Y_\bullet}(D)$ for all big divisors $D$.
One main ingredient of such a proof is the fact that the divisors which occur in the construction of Newton-Okounkov bodies with respect to the above flags, admit integral Zariski decompositions. Another one is the fact that $-K_X+Y_1$ is big and nef, which shows that the restriction morphism of every nef divisor on $X$ to the curve $Y_1$ is surjective.
After one has established such a fact, it is necessary to consider what happens when $D$ moves to the boundary of the effective cone. We will prove that the numerical and the  valuative Newton-Okounkov body in this case coincide. Then the above statement will follow from Gordan's lemma.

We end the article with two examples, which illustrate our results.
\section*{Acknowledgement}
The author would like to thank Henrik Sepp\"anen for valuable discussions and helpful suggestions for improving this article.
\section{Preliminaries}
In this section we briefly describe normal affine semigroups, the construction of Newton-Okounkov bodies, its connection to toric degenerations and   introduce the notion of Ehrhart polynomials.
Note that all varieties mentioned in this article will be defined over an algebraically closed field $k$ of characteristic $0$. Moreover, a divisor will always mean a Cartier divisor.
\subsection{Normal affine semigroups}\label{sectionnormalaffine}
In this article we will only consider semigroups contained in $\mathbb{N}^d$. So whenever, we talk about about a semigroup, we mean a set $\Gamma\subset \mathbb{N}^d$ which is closed under addition.

An \emph{affine semigroup} is a semigroup which is finitely generated.
We denote the group generated by a semigroup $\Gamma$ by $G(\Gamma)$. 
We call the semigroup $\Gamma$ a \emph{normal semigroup} if for all $g\in G(\Gamma)$ and $n\in \mathbb{N}$ such that $n\cdot g\in \Gamma$, it follows that $g\in \Gamma$. Equivalently this means that $\text{Cone}(\Gamma)\cap G(\Gamma)=\Gamma$. For more details on normal semigroups we refer to \cite[2.B]{BG09}.

When $D$ is a big divisor on a $d$-dimensional variety, and $Y_\bullet$ is an admissible flag on $X$, we know that $G(\Gamma_{Y_\bullet}(D))=\mathbb{Z}^d$ (see \cite[Lemma 2.2]{LM09}).
In this case, $\Gamma_{Y_\bullet}(D)$ is normal if all integral points of $\text{Cone}(\Gamma_{Y_\bullet}(D))$ are valuation points.

The connection to algebraic geometry comes with the fact that an affine semigroup $\Gamma$ is normal if and only if the algebra $k[\Gamma]$ is normal (see \cite[Lemma 4.39]{BG09}).
Furthermore, the projective variety $X=\proj(k[\Gamma])$ is projectively normal if $k[\Gamma]$, and thus $\Gamma$ is normal.
However, $X$ is normal if and only if there is an $m\in \mathbb{N}$ such that the $k[\Gamma]^{(m)}:= \bigoplus_{k\in \mathbb{N}} k[\Gamma]_{mk}$ is normal. 
But one can easily see that $k[\Gamma]^{(m)}=k[m\Gamma]$. 
Thus $\proj(k[\Gamma])$ is normal if and only if there is an integer $m$ such that $m\Gamma$ is normal.

Again, if $\Gamma=\Gamma_{Y_\bullet}(D)$, the variety $\proj(\Gamma)$ is normal if and only if after passing to an $m$-th multiple of $D$, all the integral points of $\text{Cone}(\Gamma_{Y_\bullet}(mD))$ are valuation points.

\subsection{Newton-Okounkov bodies}
Let $X$ be a $d$-dimensional projective variety and $D$ a big divisor.
We consider $\mathbb{Z}^d$ as an ordered group by choosing the lexicographical order.
Let 
\begin{align}
\nu\colon \bigsqcup_{D\in \text{Pic}(X)}H^0(X,\mathcal{O}(D)) \setminus \{0\}\to \mathbb{Z}^d
\end{align}
be a \emph{valuation-like function}. This is a function having the following properties:
\begin{itemize}
\item
$ \nu(f+g) \geq  \operatorname{min} \{\nu(f),\nu(g)\}$ for $f,g\in H^0(X,\struc_X(kD))$
\item $\nu(f\otimes g)=\nu(f)+\nu(g)$ for $f\in H^0(X,\struc_X(m_1D))$ and $g\in H^0(X,\struc_X(m_2D)$.
\end{itemize}
Additionally, we also pose the following conditions on $\nu$.
\begin{itemize}
\item $\nu$ has one dimensional leaves (see \cite[Section 2]{KK12} for more details)
\item The group generated by $\{ (\nu(f),k) \ | \ k\in\mathbb{N}, \ f\in H^0(X,\struc_X(kD))\}$ is equal to $\mathbb{Z}^{d+1}$.
\end{itemize}

Then we define the semigroup 
\[
\Gamma_{\nu}(D):= \{ (\nu(f),k) \ | \ k\in\mathbb{N}, \ f\in H^0(X,\struc_X(kD))\}\subseteq \mathbb{Z}^d\times \mathbb{N}.\]
The \emph{Newton-Okounkov body} of $D$ with respect to $\nu$ is given by
\begin{align}
\Delta_{\nu}(D)=\overline{\operatorname{Cone}(\Gamma_{\nu}(D))}\cap \left(\mathbb{R}^d\times \{1\}\right).
\end{align} 
In this article, we are mainly interested in valuation-like functions induced by flags $Y_\bullet$, which we denote by $\nu_{Y_\bullet}$. For
details on their construction we refer to \cite{LM09}.
\subsection{Toric degenerations}
The connection between toric degenerations and Newton-Okounkov bodies was first established in \cite{A13}. Before phrasing the main result of interest let us make explicit what we mean by a toric degeneration.
\begin{defi}
Let $X$ be projective variety. Let $D$ be a very ample divisor on $X$. We say that $X$ \emph{admits a toric degeneration with respect to $D$} if there is a flat projective family $p\colon \mathcal{X}\to \mathbb{A}^1$ such that the zero fiber $X_0:=p^{-1}(0)$ is a toric variety and $\mathcal{X}\setminus X_0$ is isomorphic to $X\times (\mathbb{A}^1\setminus \{0\})$.
Furthermore, there is a divisor $\mathcal{D}$ on $\mathcal{X}$ such that it restricts on fibers $X_t\cong X$ for $t\neq 0$ to the divisor $D$ and on $X_0$ to an ample divisor $D_0$.
We call it a \emph{normal toric degeneration} if $X_0$ is normal. We call it a \emph{projectively normal toric degeneration} if $D_0$ is very ample and $X_0$ is projectively normal with respect  to the embedding given by $D_0$.
\end{defi}

The main result in \cite{A13} can be summarized in the following theorem.
\begin{thm}[\cite{A13}]Let $X$ be a projective variety, $D$ a very ample  divisor, and $\nu$ a valuation-like function. Assume that the semigroup $\Gamma=\Gamma_{\nu}(D)=\{(\nu(s),k) \ | \ s\in H^0(X,\struc_{X}(kD)),\ k\in\mathbb{N}\}$ is finitely generated. Then there exists a toric degeneration of $X$ with respect to $D$  to the toric  variety
$X_0:= \proj (k[\Gamma])$. Moreover, the normalization of $X_0$ is the normal toric variety corresponding to the polytope $\Delta_{\nu}(D)$.
\end{thm}

For the sake of clarity we want to make it precise what it means that a Newton-Okounkov body induces a normal toric degeneration.

\begin{defi}
Let $X$ be projective variety. Let $D$ be a big divisor on $X$ and $\nu$ a valuation-like function. We say that $\Delta_{\nu}(D)$ 
\emph{induces a  toric degeneration} if $\Gamma_{\nu}(D)$ is finitely generated.
We say it induces a \emph{normal toric degeneration} if in addition $\proj(k[\Gamma_\nu])$ is normal.
\end{defi}

\subsection{Ehrhart theory}
Let $\Delta \subseteq \mathbb{R}^d$ be a convex body with non empty interior. We define the \emph{Ehrhart function} $h_{\Delta}\colon \mathbb{N}\to \mathbb{N}$ by setting 
\begin{align}
h_\Delta(k):= \vert \left(k \Delta \cap \mathbb{Z}^d\right) \vert.
\end{align} 
Now, let $\Delta\subseteq \mathbb{R}^d$ be a lattice polytope, i.e. a polytope with integral extreme points. Then there is a polynomial $P_{\Delta}=\sum_{i=0}^{d} a_it^i\in \mathbb{C}[t]$ such that $P_{\Delta}(k)=h_\Delta(k)$ for all $k\in \mathbb{N}$.
We call $P_{\Delta}$ the \emph{Ehrhart polynomial} corresponding to $\Delta$.
Some basic facts are the following:
\begin{itemize}
\item The degree of $P_\Delta$ is $d$.
\item $a_d$ is equal to $\operatorname{vol}(\Delta)$.
\item We have $a_0=1$.
\item Let $F$ be a facet of $\Delta$, and let $L_{F}$ be the induced lattice on that facet. Let furthermore $\operatorname{vol}(F)$ be the volume of $F$ with respect to the lattice $L_{F}$.
Then $a_{d-1}$ is equal to half the sum of $\operatorname{vol}(F)$ over all facets $F$ of $\Delta$.
\end{itemize}

\section{Newton-Okounkov bodies and normal toric degenerations}
In this section we want to establish the connection between the Ehrhart polynomial of $\Delta_{Y_\bullet}(D)$ and normal toric degenerations induced by $\Delta_{Y_\bullet}(D)$.

\subsection{Normal defect}
As we have already mentioned, the toric variety $X_0=\proj(k[\Gamma])$ is not necessarily normal. In order to measure the failure of normality, we introduce the following.
\begin{defi}\label{defdef}
Let $X$ be a projective variety, $D$ a big divisor on $X$ and $\nu$ a valuation-like function. Let $h_D\colon \mathbb{C}\to \mathbb{C}$ be the Hilbert function of $D$, i.e. $h_D(k)=\dim \left( H^0(X,\struc_{X}(kD))\right)$ for $k>0$. Let $\Delta_{\nu}(D)$ be the Newton-Okounkov body, and $h_{\Delta_{\nu}(D)}\colon \mathbb{Z}\to \mathbb{Z}$ the corresponding Ehrhart function, i.e. 
$h_{\Delta_{\nu}(D)}(k)=\vert k \Delta_{\nu}(D)\cap \mathbb{Z}^d \vert$.
We call the function 
\begin{align}
\operatorname{Def}_{\nu,D}:=(h_{\Delta_{\nu}(D)}-h_D) \colon \mathbb{N}\to \mathbb{N}.
\end{align}
the \emph{normal defect}.
\end{defi}
The next theorem justifies the name normal defect.
\begin{thm} \label{thmnormaldef}
Let $X$ be a projective variety, $D$ a very ample divisor on $X$ and  $\nu$ a valuation-like function. Then a rational polyhedral Newton-Okounkov body 
$\Delta_{\nu}(D)$ induces a normal toric degeneration if and only if $\operatorname{Def}_{\nu,kD}=0$ for $k\gg 0$ divisible enough. 
\end{thm}
\begin{proof}
Suppose first that $\Delta_{\nu}(D)$ induces a normal toric degeneration. This means in particular that the semigroup $\Gamma:=\Gamma_{\nu}(D)$ is finitely generated. Suppose $\Gamma$ is generated in degree $k$. Hence, we can compute $\Delta_{\nu}(kD)$ by taking the convex hull of $\Gamma_k$.
By increasing $k$ even more, we might assume that $k\Gamma=\Gamma_{\nu}(kD)$ is a normal affine semigroup. This means that all integral points in $C:=\operatorname{Cone}(k\Gamma)$ are indeed valuation points, i.e. lie in $k\Gamma$. Consider all the integral points of $C$ at level $m$. They can be identified with integral points in $m\Delta_{\nu}(kD)$. There exists $h_{\Delta_{\nu}(kD)}(m)$ many of them. However, the number of different valuation points in $k\Gamma$ of level $m$ is equal to $\dim (H^0(X,\struc_X(mkD))=h_{kD}(m)$. By the assumption that $k\Gamma$ is normal, they both agree. This proves the vanishing of the normal defect.

Now let $k \in \mathbb{N}$ such that the normal defect $\operatorname{Def}_{\nu,kD}$ is zero.  As in the previous case it follows that for each $m\in \mathbb{N}$, there are $h_{\Delta_{\nu}(kD)}(m)=\dim H^0(X,\mathcal{O}_{X}(kmD))$ integral points in the $m$-th level of $k\Gamma$.
This proves that all these integral points are valuative, i.e.
$\operatorname{Cone}(k \Gamma)\cap \left(\mathbb{Z}^{d}\times \{m\} \right)=(k\Gamma)_m$. Hence, by Gordan's lemma, $k\Gamma$ is  a normal affine  semigroup. This proves the claim. 
\end{proof}

Let us now denote by $P_D$ the Hilbert polynomial corresponding to the ample divisor $D$. This means that $P_D$ is the polynomial such that 
$P_D(k)=h_D(k)$ for $k\gg 0$. 
\begin{cor}
Let $X$, be a projective variety, $D$ a very ample divisor on $X$ and  $\nu$ a valuation-like function. Then an integral polyhedral Newton-Okounkov body 
$\Delta_{\nu}(D)$ induces a normal toric degeneration if and only if $P_{\Delta_{\nu}(D)}=P_{D}$. 
\end{cor}
\begin{proof}
This follows from the above Theorem and the fact that $h_D(k)=P_D(k)$ and $h_{\Delta_{\nu}(D)}(k)=P_{\Delta_{\nu}(D)}(k)$ for $k\gg 0$.
\end{proof}


The next two corollaries demonstrate that the condition that $\Delta_{\nu}(D)$ induces a normal toric degeneration, is completely determined by the class of $D$ and the shape of $\Delta_{\nu}(D)$.

\begin{cor}
Let $X$ be a projective variety, $Y_\bullet$ an admissible flag, and $D$ and $D^\prime$  be two numerical equivalent ample line bundles on $X$. Then $\Delta_{\nu}(D)$ induces a normal toric degeneration if and only if $\Delta_{\nu}(D^\prime)$ does.  
\end{cor}
\begin{proof}
First of all, the Newton-Okounkov body of a divisor depends only on its class \cite[Proposition 4.1]{LM09}.
Moreover, it follows from Hirzebruch-Riemann-Roch that the Hilbert polynomial of an ample divisor also depends only on the numerical class.
Hence, the normal defect of $kD$, does only depend on the numerical class for $k\gg 0$.
\end{proof}

\begin{cor}
Let $X$ be a projective variety, $\nu$ and $\nu^\prime$ be valuation-like functions, and $D$ an ample divisor on $X$. Suppose $\Delta_{\nu}(D)=\Delta_{\nu^\prime}(D)$. Then $\Delta_{\nu}(D)$ induces a normal toric degeneration if and only if $\Delta_{\nu^\prime}(D)$ does.
\end{cor}
\begin{proof}
Also follows from the equality of defects $\operatorname{Def}_{\nu,kD}=\operatorname{Def}_{\nu^\prime,kD}$ for each $k\in\mathbb{N}$.
\end{proof}

\begin{remark}
The above corollary a posteriori legitimates to say that $\Delta_{\nu}(D)$ induces a normal toric degeneration, instead of $\Gamma_{\nu}(D)$.
\end{remark}
\subsection{Normalized surface area}
Despite the characterization of normal toric degenerations in terms of the normal defect, it is not quite practical, since it involves knowing the Hilbert polynomial of a line bundle, as well as the Ehrhart polynomial. 
In this section we want to omit both problems, but still find a necessary condition to induce normal toric degenerations.

Let us fix an ample divisor $D$ on $X$. Our aim is to find a valuation-like function $\nu$ which induces a normal toric degeneration. The idea is to regard this problem as an optimization problem of the shape of $\Delta_{\nu}(D)$.

For this purpose consider the following definitions.
Let $P$ be a lattice polytope in $\mathbb{Z}^d$. Then denote by $A(P)$  the surface area of $P$ i.e. the sum of the volume of each facet $F$ with respect to the induced sublattice on $F$.

\begin{defi}
Let $X$ be a projective variety of dimension $d$, $D$ a very ample divisor on $X$ and $\nu$ a valuation-like function. Let furthermore $\Delta_{\nu}(D)$ be rational polyhedral. Let $k\in \mathbb{N}$ be an integer such that $k\Delta_{\nu}(D)$ is an integral polyhedron. Then we call 
\begin{align}
S(D,\nu):= \frac{A(\Delta_{\nu}(kD))}{k^{d-1}}
\end{align}
the \emph{normalized surface area of} $\Delta_{\nu}(D)$.
If $\Delta_{\nu}(D)$ is not rational polyhedral, we define $S(D,\nu)=\infty$.
\end{defi}

It is not a priori clear that the above definition is well defined. So let $k,k^\prime$ be two integers such that $k\Delta_{\nu}(D)$ and $k^\prime\Delta_{\nu}(D)$ are integral polyhedra. Consider both Ehrhart polynomials $P_{k\Delta_{\nu}(D)}=\sum_{i=0}^d a_i t^i$ and $P_{k^\prime\Delta_{\nu}(D)}=\sum_{i=0}^d a^\prime_i t^i$.
From our discussion of Ehrhart theory it follows that
\begin{align}
A(\Delta_{\nu}(kD))=2\cdot a_{d-1} \quad A(\Delta_{\nu}(k^\prime D))=2\cdot a^\prime_{d-1}
\end{align}
 Trivially, $P_{k\Delta_{\nu}(D)}(k^\prime)=P_{k^\prime\Delta_{\nu}(D)}(k)$.
Let us consider the Ehrhart polynomial  $P_{k\cdot k^\prime\Delta_{\nu}(D)}=\sum_{i=0}^db_it^i$.

Comparing coefficients, we deduce that $b_{d-1}=a^\prime_{d-1}\cdot k^{d-1}=a_{d-1}\cdot (k^\prime)^{d-1}$.
This proves that \begin{align}
\frac{A(\Delta_{\nu}(kD))}{k^{d-1}}= \frac{A(\Delta_{\nu}(k^\prime D))}{(k^\prime)^{d-1}}.
\end{align}

\begin{thm}\label{thmnormsurfarea}
Suppose that $\Delta_{\nu}(D)$ induces a normal toric degeneration. Then the normalized surface area $S(D,\nu)$ is minimal, i.e. for all valuation-like functions $\nu^\prime$ we have
\begin{align}
S(D,\nu^\prime)\geq S(D,\nu).
\end{align}
\end{thm}
\begin{proof}
Suppose $\Delta_{\nu}(D)$ induces a normal toric degeneration. Let $\nu^\prime$ be another valuation-like function. By Theorem \ref{thmnormaldef}, there is a $k\in\mathbb{N}$ such that the normal defect $\operatorname{Def}(kD,\nu)=0$.
We can assume that $\Delta_{\nu^\prime}(D)$ is rational polyhedral, since otherwise $S(D,\nu)=\infty$.
 Assume furthermore without loss of generality that $\Delta_{\nu}(kD)$ and $\Delta_{\nu^\prime}(kD)$ are integral polyhedra.
Since $\operatorname{Def}(kD,\nu^\prime) \geq 0$ we can follow that  
\begin{align}
\sum_{i=0}^d a_it^i=P_{\Delta_{\nu^\prime}(kD)}\geq P_{\Delta_{\nu}(kD)}=\sum_{i=0}^d b_it^i 
\end{align} 
The first coefficients $a_d$ and $b_d$ of the above polynomials are  both equal to $\operatorname{vol}(\Delta_{\nu}(kD))=\operatorname{vol}(\Delta_{\nu^\prime}(kD))=d!\cdot k^d\operatorname{vol}(D)$.
Thus, we have $a_{d-1}\geq b_{d-1}$, which in turn implies $S(D,\nu^\prime)\geq S(D,\nu)$.

\end{proof}
\section{Normal toric Degenerations on Surfaces}
In this section we want to apply the above discussions to the case where $X$ is a surface.
We will also restrict our attention to valuations coming from flags. 
One reason why the surface case in a lot of situations works particularly well is that we have a Zariski decomposition of divisors.
In our case this  leads to a nice characterization of Newton-Okounkov bodies, which makes things more explicit to handle.

Before we dive into normal toric degenerations, we give an overview of the main facts about Zariski decomposition and Newton-Okounkov bodies on surfaces  in the first two paragraphs. After that we will prove that for surfaces satisfying condition  $(*)$ (see Definition \ref{defstar})  there exists a flag $Y_\bullet$ such that its normalized surface area is minimal with respect to all admissible flags.
If we make some more assumptions on the surface $X$, we will establish an algorithm that computes for a given divisor $D$ a flag $Y_\bullet$ which induces a Newton-Okounkov body with minimal normalized surface area with respect to all valuations coming from flags. Hence, if there exists a flag which induces a normal toric degeneration, this algorithm will indeed find it.

In the following let $X$ always denote a smooth surface.
\subsection{Zariski decomposition}
Let $X$ be a smooth surface. Then the Zariski decomposition of a pseudo-effective $\mathbb{Q}$-divisor is given by $D=P+N$ where $P$ and $N$ are $\mathbb{Q}$-divisors such that
\begin{enumerate}
\item $P$ is nef
\item the support of  $N=\sum_{i=1}^{N} a_i C_i$ consists of negative curves such that $P\cdot C_i=0$ for all $i=1,\dots,N$ and 
\item the intersection matrix $(C_i\cdot C_j)_{i,j=1,\dots,N}$ is negative-definite.
\end{enumerate}
A decomposition with the above prescribed property is unique and we call $P$ the positive and $N$ the negative part of $D$.
One consequence of the above properties is that for $k\in \mathbb{N}$ divisible enough such that $kD$ as well as $kP$ are integral divisors the natural morphism
\begin{align}
H^0(X,\struc_X(kP))\to H^0(X,\struc_X(kD))
\end{align}
is an isomorphism.
That means that, after passing to a multiple, all sections of $kD$ are induced by sections of a nef divisor.
Zariski's original proof relied on the construction of the negative part, which was rather complicated. An easier approach was introduced by Bauer \cite{B09}, whose idea was to construct the positive part of an effective divisor $D$ as the maximal nef subdivisor of $D$. This reduces the problem of finding the Zariski decomposition of a given divisor to solving a linear program. More concretely, if we write  $D=\sum a_i C_i$ as a positive combination of prime divisors, one finds $P=\sum b_iC_i$, where the $b_i$ are chosen such that $\sum b_i$ is maximal under the constraints that $0\leq b_i\leq a_i$, and $\sum b_iC_i$ is nef.  

\begin{remark}
Note that even if $D$ is an integral divisor the Zariski decomposition $D=P+N$ is still a decomposition of $\mathbb{Q}$-divisors, i.e. $P$ and $N$ are not necessarily integral.
\end{remark}

However, in \cite{BPS15} the authors give an upper bound for the size of the denominators occurring in terms of the negativity of $N$. In the proof of Theorem 2.2  they show the following:
\begin{thm}[\cite{BPS15}]\label{thmzarintegral}
Let $X$ be a smooth projective surface with Picard number $\rho(X)$, let  $D$ be a divisor and $N=\sum a_i \cdot C_i$ be its negative part, with $a_i>0$ and $C_i$ prime divisors. Let furthermore $d$ be the denominator of $N$ and $b$ be the maximum of the negative numbers $(C_i)^2$.
Then we have
\begin{align}d\leq b^{\rho(X)-1}.
\end{align}

\end{thm}

Another very important feature about the Zariski decomposition, is that it induces a decomposition of the big cone into chambers $\mathcal{C}_i$; the so called \emph{Zariski chambers}. This chamber decomposition was introduced in \cite{BKS04}. We summarize some facts about this decomposition: 
\begin{itemize}
\item The support of the negative parts of $D\in \mathcal{C}_i$ for a fixed $i$ is constant.
\item The $\mathcal{C}_i$ are locally polyhedral and form a locally finite decomposition of the big cone
\item Inside the closure of each Zariski chamber $\mathcal{C}_i$ the Zariski decomposition varies linearly.
\end{itemize} 
\subsection{Newton-Okounkov bodies on surfaces}\label{sectionOkounkov}
Newton-Okounkov bodies are in general difficult to compute. However, on a surface with a valuation-like function coming from a flag, we can give a rather explicit description.
Let $Y_\bullet=(X\subseteq C\subseteq \{P\})$ be an admissible flag, i.e. $P$ is a point and $C$ is an irreducible curve which is smooth at $P$. Then we can define a valuation-like function $\nu_{Y_\bullet}$, by setting for a section $s\in H^0(X,\struc_X(D))$
\begin{align}
\nu_1(s)=\operatorname{ord}_C (s) \quad \nu_2 (s)=\operatorname{ord}_{ \{P\} } (\tilde{s})
\end{align}
where $\tilde{s}$ is the restriction of the section $s/(s_C)^{\nu_1(s)}$ to the curve $C$ and $s_C$ is a defining section of $C$.

In order to describe the Newton-Okounkov body of a big divisor $D$ with respect to a flag $C\supset \{P\}$ we fix the following notation:
\begin{itemize}
\item $\nu:=\operatorname{ord}_{C}(N)$
\item $\mu:=\sup \{t\in\mathbb{R}_{\geq 0} \ | \ D-tC \text{ is effective} \}$
\item For $t\in [0,\mu]$ we define $D_t:=D-tC=P_t+N_t$ where the latter is its Zariski decomposition.
\item We define the functions $\alpha,\beta\colon [\nu,\mu]\to \mathbb{R}_{\geq 0}$ by setting
\begin{align}
\alpha(t):= ord_{P} ({N_t}_{|C}) \quad \beta(t):= \alpha(t)+ (P_t\cdot C).
\end{align}
Moreover, we write $\alpha_D,\beta_D$ if we want to stress that we consider the divisor $D$.
\end{itemize}
Finally, we present the description of Newton-Okounkov bodies in the following theorem, which is based on the discussions in \cite[Section 6.2]{LM09} and \cite[Section 2]{KLM12}. 
\begin{thm}\label{thmokounkovsurface}
The Newton-Okounkov body of a big divisor $D$ with respect to an admissible flag $Y_\bullet$ on a surface $X$ is given by
\begin{align}
\Delta_{Y_\bullet}(D)=\{(t,y)\in \mathbb{R}^2 \ | \  t\in [\nu,\mu], \  y\in [\alpha(t),\beta(t)] \}. 
\end{align}
Moreover, $\Delta_{Y_\bullet}(D)$ is a finite polygon, with all extremal points rational except for possibly $(\mu,\alpha(\mu))$ and $(\mu,\beta(\mu))$.
\end{thm}

The proof of the above theorem uses the fact that the Zariski decomposition varies linearly inside the Zariski chambers. The fact that it is a finite polygon follows by showing that the set of divisors $D_t$ for $t\in [\nu,\mu]$ only meets finitely many chambers.
Additionally, it follows from the proof  that the extreme points of $\Delta_{Y_\bullet}(D)$ are all of the following form:
\begin{itemize}
\item $(\nu,\alpha(\nu))$, $(\nu,\beta(\nu))$
\item $(\mu,\alpha(\mu))$, $(\mu,\beta(\mu))$
\item $(t,\alpha(t))$, $(t,\beta(t))$ for $t\in (\nu,\mu)$ such that $D_t$ lies on the boundary of a Zariski chamber.
\end{itemize}

\subsection{Existence of Newton-Okounkov bodies with minimal normalized surface area}
In this paragraph we will prove that for a given divisor $D$ there exists a flag  $Y_\bullet$ such that the normalized surface area of $\Delta_{Y_\bullet}(D)$ is minimal with respect to all admissible flags.

We will now consider surfaces with the following constraints.
\begin{defi}\label{defstar}
We say that a smooth projective surface $X$ satisfies condition $(*)$ if it satisfies the following conditions:
\begin{enumerate}
\item Every pseudo-effective divisor $D$ is semi-effective, i.e. a multiple of $D$ is effective.
\item $X$ contains only finitely many negative curves. 
\end{enumerate}
\end{defi}
\begin{remark}
A large class of examples which satisfy condition $(*)$ are Mori dream surfaces.
\end{remark}

One necessary condition on the curve of the flag to induce a normal toric degeneration is the following.
\begin{lem}
Let $Y_\bullet=(C\supseteq \{P\})$ be an admissible flag such that $\Delta_{Y_\bullet}(D)$ induces a normal toric degeneration. Then the genus of $C$ is zero, i.e. $C\cong \mathbb{P}^1$.
\end{lem}
\begin{proof}
Choose a rational $t\in \mathbb{Q}$ such that the slice $\{t\}\times \mathbb{R}^{d-1}$ meets the interior of $\Delta_{Y_\bullet}(D)$.
Let then $k\in \mathbb{N}$ be such that $kD_t=kP_t+kN_t$ is a decomposition of integral divisors and $kt$ is integral. The slice $\Delta_{Y_\bullet}(kD)_{\nu_1=kt}$ contains
$k(P_t\cdot C)+1$ integral points. The valuation points having $kt$ as first coefficient are given by the image of
\begin{align*}
\operatorname{ord}_{P}\colon  H^0(X,\mathcal{O}(kD_t))_{|C}\to \mathbb{Z}
\end{align*}
and the number of valuation points is given by 
\begin{align}
h^0(X,\mathcal{O}(kD_t))_{|C}=h^0(X,\mathcal{O}(kP_t))_{|C}\leq h^0(C,\mathcal{O}_{C}(kP_t)).
\end{align}
However, it follows from Riemann-Roch on curves that for $k\gg 0$ we can compute 
\begin{align}
h^0(C,\mathcal{O}_{C}(kP_t))=k(P_t\cdot C)+1-g
\end{align}

where $g$ is the genus of $C$. But since all integral points of $\Delta_{Y_\bullet}(kD)$ are valuative for $k\gg 0$ it follows that $g=0$ and thus $C\cong \mathbb{P}^1$.
\end{proof}

We continue by proving two helpful lemmata.
\begin{lem}\label{lempoint}
Let $D$ be a big divisor on $X$ and $[C]\in N^{1}(X)$ be the numerical class of an irreducible curve $C$. Then the set of Newton-Okounkov bodies $\Delta_{Y_\bullet}(D)$ where $Y_\bullet$ is a flag such that $[Y_1]=[C]$ is finite.
\end{lem}
\begin{proof}
Consider the negative part $N_\mu$ of the pseudo-effective divisor $D_\mu=D-\mu C$. Let $C_1,\dots,C_l$ be the irreducible curves in the support of $N_\mu$. It follows from \cite[Proposition 2.1]{KLM12} that the irreducible components of $N_t$ for $t\in [\nu,\mu]$ is a subset of  $\{C_1,\dots,C_l\}$, and that $C$ is not equal to $C_i$ for all $i=1,\dots l$.
Let $\nu\leq t_1\leq \dots \leq t_{r}\leq \mu$ be all rational numbers in $[\nu,\mu]$ such that $D_{t_i}$ lies on the boundary of some Zariski chamber. By the discussion in section \ref{sectionOkounkov}, these are indeed finitely many.
Consider the negative parts $N_\nu,N_{t_1},\dots, N_{t_r}, N_{\mu}$. By replacing $D$ with $kD$ for $k\gg 0$, we may assume without loss of generality that all these negative parts are integral divisors and the numbers $t_1,\dots,t_r$ are integral.
For each $P\in C$ we have
\begin{align} 
 \alpha(\mu)=\operatorname{ord}_P N_{\mu|C}\leq \sum_{x\in C\cap N_{\mu}} \operatorname{ord}_x(N_{\mu|C})=(N_\mu\cdot C).
\end{align} 
By \cite[Theorem B]{KLM12}, the function $\alpha$ is increasing, and piecewise linear with possible breaking points at $t_1,\dots,t_r$. This shows that for a fixed class $[C]$ the function $\alpha$ is bounded by some constant independent from the point $P$. By construction, $\alpha$   takes integer values at the points $\nu,t_1,\dots,t_r,\mu$. Varying the point $P$, there are only finitely many possibilities for $\alpha$ since it is uniquely defined by its values on $\nu,t_1,\dots,t_r$ and $\mu$. But, by definition, the same holds for $\beta$. However, we have seen in Theorem \ref{thmokounkovsurface} that  $\Delta_{Y_\bullet}(D)$ is determined by $\alpha$ and $\beta$. This shows the claim.
\end{proof}
\begin{lem}\label{lemnefcurve}
Let $X$ be a smooth surface that satisfies condition $(*)$.
Let $D$ be a big and nef divisor on $X$. Then the set
\begin{align}
H_k:=\{[D^\prime]\in N^1(X)_{\mathbb{R}} \ : \ D^\prime \text{ is nef and } (D^\prime\cdot D)= k \}
\end{align}
is compact for all $k\in \mathbb{N}$.
\end{lem}
\begin{proof}
Suppose that $H_k$ is not compact. It is easy to check that $H_k$ is closed. This means $H_k$ is not bounded. But since it is also convex, there exists for every point in $H_k$  a half line which is completely contained in $H_k$.
For this purpose fix any ample class $[A]\in N^1(X)$ and consider the $\mathbb{Q}$-divisor $A^\prime=\frac{k}{(D\cdot A)}A$ which lies in $H_k$. Since $H_k$ is not bounded there is a divisor class $[F]\in N^1(X)_\mathbb{R}$ such that for all $\lambda>0$, the class $[A^\prime+\lambda F]$ lies in $H_k$.
We claim that $F$ is nef. Indeed, suppose that $F$ is not nef. Then for $\lambda\gg 0$ the divisor $A^\prime+\lambda F$ is not nef as well and thus does not lie in $H_k$. 

For a given $\lambda>0$, we have 
\begin{align}
(A^\prime+\lambda F)^2=(A^\prime)^2+2\lambda(A^\prime\cdot F) +(F)^2\geq (A^\prime)^2 +2\lambda (A^\prime\cdot F).
\end{align} 
But $A^\prime$ is ample and $F$ nef, hence semi-effective by condition $(*)$. This implies that $A^\prime\cdot F >0$ and  enables us to find a $\lambda>0$ such that 
\begin{align}
\sqrt{(A^\prime+\lambda F)^2}> k/\sqrt{(D)^2}.
\end{align}

 Here, we used the fact that $D$ is big and nef and thus $(D^2)>0$.
As $D$ and $A+\lambda F$ are both nef, we can use the Hodge index theorem to deduce
\begin{align}
(D\cdot (A^\prime+\lambda F))\geq \sqrt{(D)^2 \cdot (A^\prime+\lambda F)^2}>k.
\end{align}
This shows that $A^\prime+\lambda F$ does not lie in $H_k$, which is a contradiction.
Hence, $H_k$ is compact.

\end{proof}

\begin{thm}\label{thmexminsurf} Let $X$ be a smooth surface satisfying condition $(*)$.
Let $D$ be a big divisor on $X$. Then there exists an admissible flag $Y_\bullet=(C\supset \{x\} )$ such that its normalized surface area $S(D,\nu_{Y_\bullet})$ is minimal, i.e. for any admissible flag $Y_\bullet^\prime$ we have $S(D,\nu_{Y_\bullet})\leq S(D,\nu_{Y_\bullet^\prime})$.
\end{thm}
\begin{proof}By scaling and considering the positive part in the Zariski decomposition of $D$, we can without loss of generality assume that $D$ is big and nef.
The idea of the proof is to show that only a finite number of classes of curves $C$ have to be tested. Together with Lemma \ref{lempoint} we can then prove the claim. 

First of all if $C$ is an irreducible curve, then its class $[C]$ is either nef or it is a negative curve depending on whether $C^2\geq 0$ or $C^2<0$. 
Since $X$ satisfies condition $(*)$, we have to test only finitely many negative curves.  
Hence, we can restrict our attention to the case that $C$ is nef.
 Let $Y_\bullet$ be any  admissible flag such that $Y_1$ is nef, and set $M:=S(D,\nu_{Y_\bullet})$. Then for all $C^\prime$, and any point $P^\prime \in C^\prime$ such that $(D\cdot C^\prime)>M$, 
 we know that there are already more than $M+1$ integral points on the boundary of $\Delta_{C^\prime\supset\{P^\prime\}}(D)$, namely $(0,\alpha(0)),(0,\alpha(0)+1),\dots,(0,\beta(0))$. However, this implies $S(D,\nu_{Y_\bullet})>M$.
Hence, we have limited the candidates to nef irreducible curves $C^\prime$ such that $(D\cdot C^\prime)\leq M$. But Lemma \ref{lemnefcurve} implies there are only finitely many integral nef classes of curves which satisfy this condition. Moreover, each class of a curve $[C]$ we have to test, has finitely many different Newton-Okounkov bodies, when varying the point $P$ by Lemma \ref{lempoint}. This proves the claim.
\end{proof}

\subsection{Algorithm for finding a flag with minimal normalized surface area}

In this paragraph we want to introduce and discuss an algorithm, which outputs for a given big divisor $D$ on a surface $X$ a flag $Y_\bullet$ such that $\Delta_{Y_\bullet}(D)$ induces a normal toric degeneration if  such a flag exists.
In Theorem \ref{thmexminsurf}, we have limited the possible candidates for flags which induce  normal toric degenerations to finitely many classes of curves $[Y_1]$. However, it is a rather difficult task to describe what possible points $P\in C$ can occur and how the function $\alpha$  from Section \ref{sectionOkounkov} varies. The idea of this section is to show that it is possible to reduce to a general point on the chosen curve.
Then $\alpha=0$ and $\beta(t)=(P_t\cdot C)$. It follows that the corresponding Newton-Okounkov body only depends on the numerical class of the curve $C$ and in this situation Theorem \ref{thmexminsurf} gives rise to a rather explicit algorithmic way of finding a class of a curve which is the best candidate for defining a flag $Y_\bullet$ such that $\Delta_{Y_\bullet}(D)$ induces a normal toric degeneration.

The price we pay for being able to choose a general point is the following constraint.
\begin{defi}
We say that a smooth projective surface $X$ satisfies condition $(**)$ if it satisfies condition $(*)$ and the Zariski decomposition is  a decomposition of integral divisors, i.e. for each integral divisor $D$ its positive part $P(D)$ as well as its negative part $N(D)$ is integral.
\end{defi}

\begin{remark}
It follows from Theorem \ref{thmzarintegral} that a surface  having only negative curves with self intersection $-1$ induces integral Zariski decompositions for all divisors. An example for this situation would be smooth del Pezzo surfaces (more details follow in the next section).
\end{remark}

The following lemma is the key for reducing to the case of a general point $P$ on $C$.
\begin{lem}\label{lemnormpoint}
Let $X$ be a smooth surfaces satisfying condition $(**)$. Suppose $D$ is a big divisor and $Y_\bullet=(C\supset \{P\})$ is an admissible flag such that $\Delta_{Y_\bullet}(D)$ induces a normal toric degeneration. Then for each point $P^\prime\in C$ consider the flag $Y_\bullet^\prime=(C\supset \{P^\prime\})$.
Then $\Delta_{Y^\prime_\bullet}(D)$ induces a normal toric degeneration as well. 
\end{lem} 
\begin{proof} 
We will do this by proving that the Ehrhart polynomial $P_{\Delta_{Y_\bullet}(kD)}$ is independent of the point $P$ for $k\gg 0$.
Let $k\gg 0$ such that $\Delta_{Y_\bullet}(kD)$ is an integral polytope.
Define for integral $m,t$ the divisor $D_{m,t}:=mD-tC=:P_{m,t}+N_{m,t}$.

Since $X$ satisfies condition $(**)$, the function $\alpha_{kD}(t)=\text{ord}_{P}(N_{k,t|C})$ and $\beta_{kD}(t)=\alpha_{kD}(t)+(P_{k,t}\cdot C)$ admit integral values for each integral $t$. From this we can deduce
\begin{align}
\vert k\Delta_{Y_\bullet}(D)\cap \mathbb{Z}^2  \vert=\sum_{t=k
\nu}^{k\mu}\left((P_{k,t}\cdot C)+1\right).
\end{align}
But the right hand side does not depend on the choice of the point $P\in C$. Hence, the result follows from Theorem \ref{thmnormaldef}.
\end{proof}
\begin{remark}
The condition that $X$ admits integral Zariski decompositions is indeed necessary for the above lemma.

\begin{figure}[H]
\includegraphics[scale=0.2]{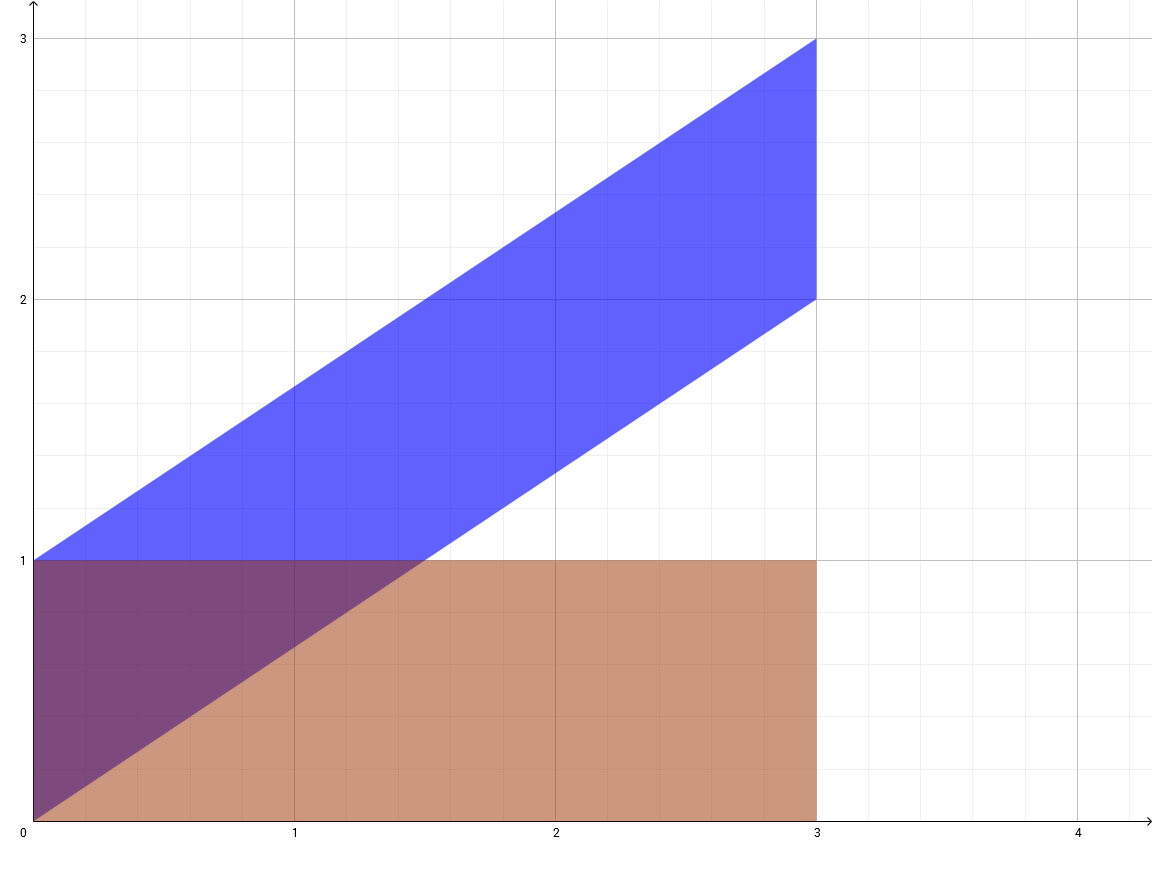}

\caption{Newton-Okounkov body of toric variety}
\label{newtontoric}

\end{figure} 
For a counterexample consider the blue (dark) polytope in Figure \ref{newtontoric}. Then, by the discussion in \cite[Section 6.1]{LM09}, there is a toric variety $X$ and  a divisor $D$ such that with respect to a certain flag $Y_\bullet$, the corresponding Newton-Okounkov body $\Delta_{Y_\bullet}(D)$ is equal to the above blue (dark) polytope. It also follows from this discussion that $\Delta_{Y_\bullet}(D)$ induces a normal toric degeneration.
If we change the point $Y_2$ of the flag and pick a general one instead, the resulting Newton-Okounkov body equals the red (light) polytope in Figure \ref{newtontoric}.
However, the number of integral boundary points on the red polytope (8) is bigger than the boundary points on the blue polytope (4).
This proves that for a general point, the corresponding Newton-Okounkov body does not induce a normal toric degeneration even though this holds for a special point.

\end{remark}

As we have seen above, the Ehrhart polynomial of $\Delta_{Y_\bullet}(D)$ only depends on the numerical class $[Y_1]$ if $X$ satisfies condition $(**)$. Hence, also the normalized surface area $S(D,\nu_{Y_\bullet})$ only depends on the numerical class $[Y_1]$. Therefore, we will just write $S(D,[Y_1])$ 
instead of $S(D,\nu_{Y_\bullet})$.

We are now able to describe an algorithm which will give us for a given divisor $D$ an optimal class of a curve $[C]$.

\begin{algorithm} 

\label{algo}
\DontPrintSemicolon
 \KwIn{a big divisor D}
 \KwResult{optimal class of a curve C}
  D=P+N \quad\tcp{compute Zariski decomposition}\;
  D:=P \quad  \tcp{replace the divisor D by its positive part}\; 
 \For{ negative classes of curves $N_i$}{
   compute $S(D,[N_i])$
   } 
   $optimum:= \min_{N_i} S(D,[N_i])$\;
   $optimalcurve:= \operatorname{argmin}_{N_i} S(D,[N_i])$\;
 \For{ all $\xi \in N^1(X)$ s.t. $(D\cdot \xi)<optimum+1$}{
  \If{$\xi==[C]$ for some irreducible curve $C$}{
  \If{$optimum\leq S(D,\xi)$}{
  $optimum=S(D,\xi)$\;
  $optimalcurve:=\xi$\;
  }
  \uElseIf{ $optimum== S(D,\xi)$}
  {$optimalcurve.append(\xi)$\;}  
  }
 
 }
 
\KwOut{optimalcurve}
 \caption{Algorithm for finding normal toric degenerations}
\end{algorithm}

\begin{thm}
Let $X$ be a projective surface satisfying condition $(**)$. Let $D$ be a very ample divisor. Let us assume that there is a flag $Y_\bullet$ such that $\Delta_{Y_\bullet}(D)$ induces a normal toric degeneration. Then the output of Algorithm 1 gives a list of all classes of curves $C_i$ which give rise to  flags such that the corresponding Newton-Okounkov bodies induce  normal toric degenerations.   
\end{thm}
\begin{proof}
This follows from Theorem \ref{thmnormsurfarea},  the proof of Theorem \ref{thmexminsurf} and Lemma \ref{lemnormpoint}.
\end{proof}

\section{Normal toric Degenerations of (weak) del Pezzo Surfaces}
In this paragraph we will use our previous findings and additional ideas to construct normal toric degenerations of \emph{(weak) del Pezzo surfaces}. 

\subsection{Normal toric degeneration of smooth del Pezzo surfaces}
Let us first present some basic facts about smooth del Pezzo surfaces.
\begin{defi}
We call $X$ a del Pezzo surface if it is a surface and its anticanonical divisor $-K_X$ is ample.
\end{defi}
Before we give the characterization of smooth del Pezzo surfaces, let us define what we mean by points in general position.
\begin{defi}
We say that $1\leq r \leq 8$ distinct points $p_1,\dots, p_8$ in $\mathbb{P}^2$ are in \emph{general position}
if:
\begin{itemize}
\item No three of them lie on a line.
\item No six of them lie on a conic.
\item No eight of them lie on a cubic with a singularity at some of the $p_i$.
\end{itemize}
\end{defi}
We can now state the well known characterization of smooth del Pezzo surfaces.
\begin{thm}
Up to isomorphy the smooth del Pezzo surfaces are  given by $\mathbb{P}^1\times \mathbb{P}^1$ or the blow-up of $\mathbb{P}^2$ in $0\leq r\leq 8$ points in general position.
\end{thm}

Let $X_r$ be the smooth del Pezzo surface obtained by blowing up $r$ points in general position.
In the following we collect some more facts, we want to use:
\begin{enumerate}
\item We have 
\begin{align}
\operatorname{Pic}(X)\cong N^1(X)\cong \mathbb{Z}^{r+1}\cong\mathbb{Z}[H]\oplus \bigoplus_{i=1}^r \mathbb{Z}[E_i]
\end{align}
where $H$ is the total transform of a line in $\mathbb{P}^2$ and $E_i$ are the exceptional divisors.
\item The intersection form on $N^1(X)$ is determined by the identities $(H)^2=1$, $(H\cdot E_i)=0$, $(E_i\cdot E_j)=-\delta_{ij}$.
\item The anticanonical divisor is given by
\begin{align}
-K_{X_r}=3H-E_1-\dots - E_r.
\end{align}
\item  Every irreducible curve with negative self intersection number is a $(-1)$-curve, and it is, up to permutation of indices, linear equivalent to one of the following divisors:
\begin{align}
&E_1\\
&H-E_1-E_2\\
&2H-E_1-\dots E_5\\
&3H-2E_1-E_2-\dots - E_7\\
&4H-2E_1-2E_2-2E_3-E_4-\dots -E_8\\
&5H-2E_1-\dots - 2E_6-E_7-E_8\\
&6H-3E_1-2E_2-\dots - 2 E_8.
\end{align}
\item
Let $N=\{C_1,\dots, C_N\}$ be the set of $(-1)$-curves.
 The effective cone $\operatorname{Eff}(X)$ is generated by the negative curves $C_i$ in $N$. The nef cone is determined by the supporting hyperplanes 
\begin{align} 
 C_i^{\perp}:=\{[D]\in N^1(X)_{\mathbb{R}} \ | \ D\cdot C_i=0\}.
 \end{align}
\item The Zariski chambers of $X_r$ are also determined by the chamber decomposition of $\operatorname{Eff}(X)$ induced from the hyperplanes $C_i^{\perp}$.
\item Suppose $r=1,\dots, 6$. A divisor class $D\in \operatorname{Pic}(X_r)$ contains an irreducible curve $C\in |D|$ if and only if   $D$ is either (a) one of the $(-1)$-curves in $N$ or (b) $D$ is big and nef or (c) $D$ is  a conic (i.e. $D\cdot (-K_{X_r})=2$) and $D^2=0$.
\item Let $C\subset X_r$ be an irreducible smooth curve such that $C\equiv_{\text{lin}}aH-b_1E_1-\dots - b_r E_r$.
Then the genus of $C$ is given by
\begin{align}
g(C)= 1/2 (a-1)(a-2)-1/2 \sum_{i=1}^{r} b_i(b_i-1).
\end{align}
\end{enumerate}
As a reference, we refer to \cite[V.4]{H77} for $(a),(b),(c),(g)$, to \cite[Chapter 5]{ADH14} for properties $(d),(e)$.
Property $(f)$ is derived in \cite[Proposition 3.4]{BKS04}. Furthermore, $(h)$ is an easy calculation using Rieman-Roch.

Let us now apply Algorithm for a specific del Pezzo surface.
\begin{ex}
Let $X_5$ be the blow-up of $\mathbb{P}^2$ in five general points. That means that no three of them lie on a line. In this case, the negative curves are of the form
\begin{align}
E_1,\dots, E_5, \text{ and }H-E_i-E_j  \text { for } i,j=1,\dots, 5, \ i\neq j .
\end{align}
For a given divisor $D$ and  a curve $C$, we have all the necessary information to compute the Newton-Okounkov body $\Delta_{C\supset \{P\}}(D)$ for a very general point $P\in C$. With the help of a computer we can thus use our algorithm to compute  the set of optimal curves, and the optimal normalized surface area for a given divisor $D$.
We can use \cite[Example 1.3]{SX10} to efficiently compute the Hilbert polynomial of a given divisor $D$ as the Ehrhart polynomial of some polytope. Hence, we can compare the second coefficient of the Hilbert polynomial with the normalized surface area. If they agree the given Newton-Okounkov body with respect to the curves found by the algorithm induce normal toric degenerations.
Running the algorithm for some randomly chosen divisors gives the following result. Note that all the divisor classes are represented by the basis $H,E_1,\dots, E_5$.
{\center
\begin{table}[H]
\scriptsize
\begin{tabular}{l|m{25em}|m{5em}|m{5em}}
\hline
\textbf{D} & \textbf{optimal curves} & \textbf{\makecell{min.\\ $S(D,[C])$} }&\textbf{\makecell{2nd coef.\\ of $h_D(t)$} } \\
\hline
$(6, -1, -1, -2, -3, -4)$ & 
\begin{tabular}{m{10em}|m{10em}}
nef curves & negative curves\\
\hline
$(4, -2, 0, -2, -2, -2)$, $(4, -2, 0, -2, -2, -2)$, $(3, -1, -1, -1, -1, -2)$, $(2, -1, -1, -1, 0, -1)$, $(2, -1, -1, 0, -1, -1)$, $(2, -1, 0, -1, -1, -1)$, $(4, 0, -2, -2, -2, -2)$, $(2, 0, -1, -1, -1, -1)$, $(2, 0, 0, -1, -1, -1)$, $(1, 0, 0, -1, 0, 0)$, $(1, 0, 0, 0, -1, 0)$, 
$(1, 0, 0, 0, 0, -1)$ &

$(0, 1, 0, 0, 0, 0)$, $(0, 0, 1, 0, 0, 0)$, $(0, 0, 0, 1, 0, 0)$, $(0, 0, 0, 0, 1, 0)$, $(0, 0, 0, 0, 0, 1)$, $(1, -1, -1, 0, 0, 0)$, $(1, 0, -1, -1, 0, 0)$, $(1, 0, 0, -1, -1, 0)$, $(1, -1, 0, -1, 0, 0)$, $(1, -1, 0, 0, -1, 0)$, $(1, 0, -1, 0, -1, 0)$, $(1, 0, 0, 0, -1, -1)$, $(1, 0, 0, -1, 0, -1)$, $(1, 0, -1, 0, 0, -1)$, $(1, -1, 0, 0, 0, -1)$, $(2, -1, -1, -1, -1, -1)$
\end{tabular}
  & 6 & 6 \\ 
\hline 
$(6, -1, -3, -1, -2, -3)$ &
\begin{tabular}{m{10em}|m{10em}}
nef curves & negative curves\\
\hline

$(2, -1, -1, -1, 0, -1)$, $(2, -1, -1, -1, 0, -1)$, $(2, -1, -1, 0, -1, -1)$, $(2, 0, -1, -1, -1, -1)$, $(1, 0, -1, 0, 0, 0)$, $(1, 0, 0, 0, 0, -1)$
&

$(0, 1, 0, 0, 0, 0)$, $(0, 0, 1, 0, 0, 0)$, $(0, 0, 0, 1, 0, 0)$, $(0, 0, 0, 0, 1, 0)$, $(0, 0, 0, 0, 0, 1)$, $(1, -1, -1, 0, 0, 0)$, $(1, 0, -1, -1, 0, 0)$, $(1, 0, 0, -1, -1, 0)$, $(1, -1, 0, -1, 0, 0)$, $(1, -1, 0, 0, -1, 0)$, $(1, 0, -1, 0, -1, 0)$, $(1, 0, 0, 0, -1, -1)$, $(1, 0, 0, -1, 0, -1)$, $(1, 0, -1, 0, 0, -1)$, $(1, -1, 0, 0, 0, -1)$, $(2, -1, -1, -1, -1, -1)$ 
\end{tabular}
 & 8
 & 8 \\ \hline
 \end{tabular} 
 \end{table}
 
\begin{table}[H]
\scriptsize
\begin{tabular}{l|m{25em}|m{5em}|m{5em}}
\hline
\textbf{D} & \textbf{optimal curves} & \textbf{\makecell{min.\\ $S(D,[C])$} }&\textbf{\makecell{2nd coef.\\ of $h_D(t)$} } \\
\hline

$ (8, -3, -2, -2, -2, -3)$ &
\begin{tabular}{m{10em}|m{10em}}
nef curves & negative curves\\
\hline
\center $\emptyset $ &
$(0, 1, 0, 0, 0, 0)$, $(0, 0, 1, 0, 0, 0)$, $(0, 0, 0, 1, 0, 0)$, $(0, 0, 0, 0, 1, 0)$, $(0, 0, 0, 0, 0, 1)$, $(1, -1, -1, 0, 0, 0)$, $(1, 0, -1, -1, 0, 0)$, $(1, 0, 0, -1, -1, 0)$, $(1, -1, 0, -1, 0, 0)$, $(1, -1, 0, 0, -1, 0)$, $(1, 0, -1, 0, -1, 0)$, $(1, 0, 0, 0, -1, -1)$, $(1, 0, 0, -1, 0, -1)$, $(1, 0, -1, 0, 0, -1)$, $(1, -1, 0, 0, 0, -1)$, $(2, -1, -1, -1, -1, -1)$ 
\end{tabular}
&12&12 \\ \hline
$(4, -1, -1, -1, 0, -1)$ &
\begin{tabular}{m{10em}|m{10em}}
nef curves & negative curves\\
\hline

$(2, -1, -1, -1, 0, -1)$, $(2, -1, -1, -1, 0, -1)$, $(1, -1, 0, 0, 0, 0)$, $(1, 0, -1, 0, 0, 0)$, $(1, 0, 0, -1, 0, 0)$, $(1, 0, 0, 0, 0, -1)$ 
&
$(0, 1, 0, 0, 0, 0)$, $(0, 0, 1, 0, 0, 0)$, $(0, 0, 0, 1, 0, 0)$, $(0, 0, 0, 0, 1, 0)$, $(0, 0, 0, 0, 0, 1)$, $(1, -1, -1, 0, 0, 0)$, $(1, 0, -1, -1, 0, 0)$, $(1, 0, 0, -1, -1, 0)$, $(1, -1, 0, -1, 0, 0)$, $(1, -1, 0, 0, -1, 0)$, $(1, 0, -1, 0, -1, 0)$, $(1, 0, 0, 0, -1, -1)$, $(1, 0, 0, -1, 0, -1)$, $(1, 0, -1, 0, 0, -1)$, $(1, -1, 0, 0, 0, -1)$, $(2, -1, -1, -1, -1, -1)$ 
\end{tabular}
&8 &
8 \\ \hline
$(7, -4, 0, -2, -3, -3)$&
\begin{tabular}{m{10em}|m{10em}}
nef curves & negative curves\\
\hline

$(3, -2, -1, -1, -1, -1)$, $(3, -2, -1, -1, -1, -1)$, $(2, -1, -1, -1, -1, 0)$, $(2, -1, -1, -1, 0, -1)$, $(2, -1, -1, 0, -1, -1)$, $(4, -2, 0, -2, -2, -2)$, $(3, -2, 0, -1, -1, -1)$, $(2, -1, 0, -1, -1, -1)$, $(2, -1, 0, -1, -1, 0)$, $(2, -1, 0, -1, 0, -1)$, $(2, -1, 0, 0, -1, -1)$, $(1, -1, 0, 0, 0, 0)$, $(1, 0, 0, -1, 0, 0)$, $(1, 0, 0, 0, -1, 0)$, $(1, 0, 0, 0, 0, -1)$] 
&
$(0, 1, 0, 0, 0, 0)$, $(0, 0, 1, 0, 0, 0)$, $(0, 0, 0, 1, 0, 0)$, $(0, 0, 0, 0, 1, 0)$, $(0, 0, 0, 0, 0, 1)$, $(1, -1, -1, 0, 0, 0)$, $(1, 0, -1, -1, 0, 0)$, $(1, 0, 0, -1, -1, 0)$, $(1, -1, 0, -1, 0, 0)$, $(1, -1, 0, 0, -1, 0)$, $(1, 0, -1, 0, -1, 0)$, $(1, 0, 0, 0, -1, -1)$, $(1, 0, 0, -1, 0, -1)$, $(1, 0, -1, 0, 0, -1)$, $(1, -1, 0, 0, 0, -1)$, $(2, -1, -1, -1, -1, -1)$ 
\end{tabular}
& 9 & 9 \\ \hline
\end{tabular} 
\end{table}
}
\normalsize

We can make several conjectures from this  example. First of all in each example the optimal  normalized surface area is equal to the second coefficient of the Hilbert polynomial of $D$. Thus, in each example we do indeed get normal toric degenerations.
Moreover, in each example all negative curves are optimal.
We will see later that this is true for all varieties $X_r$ for $r=1,\dots, 6$.
\end{ex}

The next theorem describes some conditions on $X$ and on the flag $Y_\bullet$ which make sure that $\Delta_{Y_\bullet}(D)$ induces a normal toric degeneration.
\begin{thm} \label{thmnormaldegminusone}
Let $X$ be a smooth surface, admitting  integral Zariski decompositions. Let $Y_\bullet=(C\supseteq \{P\})$ be an admissible flag such that  $C\cong \mathbb{P}^1$ and $-K_{X}-C$ defines a big and nef class. Let $D$ be a big divisor on $X$. Then $\Delta_{Y_\bullet}(D)$ induces a normal toric degeneration if and only if $\Delta_{Y_\bullet}(D)$ is rational polyhedral.
 \end{thm}
\begin{proof}
Let us first make some observations.
If $D$ is a nef divisor, then by assumption $D-C-K_X$ is big and nef. We can therefore use the Kawamata-Viehweg vanishing theorem to deduce that $H^1(X,\struc_{X}(D-C))=0$.
This implies that the restriction morphism $H^0(X,\struc_X(D))\to H^0(C,\struc_C(D))$ is surjective for every nef divisor $D$.
Our aim is to show that we have an equality
\begin{align}
\Gamma_k(D)=k\cdot \Delta_{Y_\bullet}(D)\cap \mathbb{Z}^2.
\end{align}
Then the statement follows by using Gordan's lemma.
We will  do this by considering the vertical $t$-slices of $k\Delta_{Y_\bullet}(D)$, i.e. points such that the first coordinate is equal to a fixed  integer $t\in [k\nu,k\mu]$. The second coordinate of the valuation points $\Gamma_k(D)$ in the $t$-slice are given by the valuation points of the restricted linear series $H^0(X,\struc_{X}(kD-tC))_{|C}$ of the valuation $\operatorname{ord}_P$.
Define $D_{k,t}:=kD-tC$, $P_{k,t}:=P(D_{k,t})$ and $N_{k,t}:=N(D_{k,t})$.
 By Theorem \ref{lemnormpoint}, we can without loss of generality assume that the point $P$ is not contained in the support  of the negative part $N_{k,t}$.  Since $X$ induces an integral Zariski decomposition,
we can replace the restricted linear series  $H^0(X,\struc_X(D_{k,t}))_{|C}$ with the linear series $H^0(X,\struc_X(P_{k,t}))_{|C}=H^0(C,\struc_C(P_{k,t}))$.
As $C\cong \mathbb{P}^1$, we can apply Riemann-Roch to deduce that 
\begin{align}
\dim H^0(C,\mathcal{O}_C(P_{t,k|C}))=(P_{t,k}\cdot C)+1.
\end{align}
 Hence, the valuation points in the $t$-slice are exactly all the points $(t,s)$ where $s\in \{0,\dots (P_{k,t}\cdot C) \}$. These are all the integer points in the $t$-slice of $k\Delta_{Y_\bullet}(D)$.
\end{proof}

We can use the above theorem to prove the following.
\begin{thm}
Let $X_r$ be the blow-up of $r$ general points in $\mathbb{P}^2$ for $r=1,\dots, 6$. Let $D$ be a big divisor on $X_r$, $C\subset X_r$ a negative curve and $P\in C$ an arbitrary point.
Then $\Delta_{C\supset \{P\}}(D)$ induces a normal toric degeneration.  
\end{thm}
\begin{proof}
Since $X_r$ has a rational polyhedral effective cone, the Newton-Okounkov body $\Delta_{Y_\bullet}(D)$ is rational polyhedral for all big divisors and admissible flags $Y_\bullet$. It follows from Theorem \ref{thmzarintegral} and the fact that the only negative curves in $X_r$ are $(-1)$-curves, that $X_r$ admits  integral Zariski decompositions.
The negative curves of $X_r$ are either exceptional divisors $E_i$, lines $H-E_i-E_j$ or conics $2H-E_{i_1}-\dots - E_{i_4}$. All of them are rational. In addition, a calculation shows that the divisor $-K_{X}-C$ is big and nef for all possible negative curves $C$. Now, we can use Theorem \ref{thmnormaldegminusone}, which proves the claim.
\end{proof}
\begin{remark}
Note that for $X_r$, where $r=7,8$ the assumptions of Theorem \ref{thmnormaldegminusone} are not fulfilled for all negative curves.
Consider for example the negative curve $C=3H-2E_1-E_2-\dots E_7$ on $X_7$. Then $-K_{X_7}-C=E_1$ which is clearly not big and nef.
\end{remark}


\subsection{Normal toric degeneration on weak del Pezzo surfaces}
In this paragraph we want to discuss examples of weak del Pezzo surfaces which induce normal toric  degenerations.
\begin{defi}
We call $X$ a \emph{weak del Pezzo surface } if it is a surface and its anticanonical divisor $-K_X$ is nef and big.
 \end{defi}
The characterization of smooth weak del Pezzo surfaces is a bit more complex. Roughly speaking, more constellation of points to blow-up are allowed.
One of the main differences to del Pezzo surfaces is that no longer only $(-1)$-curves occur as negative curves but also $(-2)$-curves.

We will focus on two examples. First, the blow-up of six points on a conic the and second, the blow-up of four points where three of them lie on a line.

\subsubsection{Blow-up of six points on a conic} 

Consider the variety $S_6$ which is given as the blow-up of six points in $\mathbb{P}^2$ such that no three of them are collinear but all six lie on a single conic. The negative curves are: 
\begin{enumerate}
\item $E_1,\dots, E_6$  the exceptional divisors
\item $H-E_i-E_j$ for $i\neq j$, $i,j\in \{1,\dots, 6\}$ the  strict transforms of the lines through two points.
\item $2H-E_1-\dots -E_6$ the strict transform of the conic through all the six points.
\end{enumerate}
The first two types of curves are $(-1)$-curves and the last one is a $(-2)$-curve.

\begin{thm}
Let $D$ be a big divisor on $S_6$. Let furthermore $C$ be the  strict transform of the conic going through the six chosen points, and $P\in C$ an arbitrary point. Then $\Delta_{C\supset \{P\}}(D)$ induces a normal toric degeneration.
\end{thm}
\begin{proof}
The proof works similar as before with the only difference that there are also $(-2)$-curves occurring. This means that it is not clear whether $S_6$ admits integral Zariski decompositions.

However, a computation shows that $-K_X-C$ is big and nef. 
We know that $C$ is not contained in the support of $N_t$ for $\nu\leq t \leq \mu$ (see proof of Proposition 2.1 in \cite{KLM12}).

 Since $C$ is the only $(-2)$-curve in $S_6$, the support of the divisors of $N_t$ only consists of $(-1)$-curves. We can thus use Theorem \ref{thmzarintegral} to deduce that if $D_t$ is integral then also $P_t$ and $N_t$ are integral. Then the proof works exactly as in Theorem \ref{thmnormaldegminusone}. \end{proof}

In order to be able to compute Newton-Okounkov bodies, we need to know the Zariski chambers of the effective cone of $S_6$.
Note that unlike in the case of del Pezzo surfaces, the decomposition of Zariski chambers is not necessarily given by the decomposition induced from the hyperplanes  $C^{\perp}$ where $C$ is in the set of negative curves. This is a consequence of the fact that there exists a $(-2)$-curve.

 In general it is quite difficult to describe this decomposition. However, in order to compute Newton-Okounkov bodies of a divisor $D$ with respect to the curve $C$ given by the conic going through the six points, we just need to compute the wall crossings of the segment $D-tC$ for $t\in [\nu,\mu]$.
The next lemma describes these crossing points.
\begin{lem}\label{lemzarint}
Let $D$ be a big divisor on $S_6$.
Then the intersections of the divisors $D-tC$ for $t\in (\nu,\mu)$  with the boundary of the Zariski chambers all lie in the set $\bigcup C_i^{\perp}$ where the union is taken over all $(-1)$-curves.
\end{lem}
\begin{proof}
The proof is very similar to Proposition $3.4$ in \cite{BKS04}. It is shown in the mentioned proof  that if $N$ is a negative divisor whose support contains only $(-1)$-curves, then all the irreducible components of $N$ are orthogonal.
Let us now suppose $D_t:=D-tC$ for $t\in (\nu,\mu)$ lies on the boundary of a Zariski chamber.
If we define for a divisor $D$ the sets
\begin{align}
\operatorname{Null}(D)&=\{ C \ | \ \text{irreducible with } (C\cdot D)=0\}\\
\operatorname{Neg}(D)&=\{C \ | \text{ irreducible component of } N(D)\},
\end{align}
 then according to \cite[Proposition 1.5]{BKS04}, 
this means that 
\begin{align}
\operatorname{Null}(P_t)\setminus \operatorname{Neg}(D_t)\neq \emptyset.
\end{align}

Let $C^\prime$ be a curve which lies in $\operatorname{Null}(P_{t})$ but not in $\operatorname{Neg}(D_t)$.
Then $C^\prime$ is a negative curve and $N_{D_t}+C^\prime$ is a negative divisor according to  \cite[Lemma 4.3]{BKS04}.
We want to show that $C^\prime\neq C$. Suppose that they are equal. Then $(P_t\cdot C)=0$. We know from the choice of $t$ that the slice $\Delta_{C\subset\{P\}}(D)_{\nu_1=t}$ has length bigger than $0$ for $t\in (\nu,\mu)$. However, this is a contradiction to $(P_t\cdot C)=0$. Hence, $C^\prime \neq C$ and thus $C^\prime$ is a $(-1)$-curve. It follows that the support of  $N_t+C^\prime$ consists of $(-1)$-curves and we conclude $(C^\prime\cdot N_{t})=0$ which implies that $(D_t\cdot C)=0$. This shows that $D_t\in C^\perp$.
\end{proof}

We are now able to present an example of a Newton-Okounkov body on $S_6$ which induces a normal toric degeneration.
\begin{ex}
Let us consider the divisor $D=4H-E_1-\dots E_6$. This is an ample divisor. 
The corresponding Newton-Okounkov body with respect to the curve $C=2H-E_1-\dots - E_6$ and a general point $P$ is illustrated in Figure \ref{exok1}.
\begin{figure}[H]

\begin{center}
\includegraphics[scale=0.5]{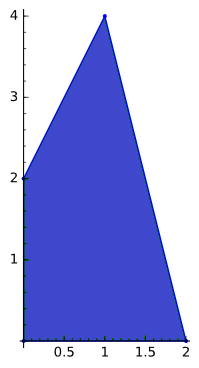}
\end{center}
\caption{N.-O. body of $D=4H-E_1-\dots, E_6$ on $S_6$}
\label{exok1}
\end{figure}
The Hilbert polynomial, which is equal to the Ehrhart polynomial of $\Delta_{Y_\bullet}(D)$, is given by
\begin{align}
P_D(t)=5t^2+3t+1.
\end{align}
\end{ex}
\subsection{Blow-up of four points three of them  on a line}
Let $L_3$ be the blow-up of $\mathbb{P}^2$ of four points where three points lie on a line. This is again a weak del Pezzo surface. The negative curves are:
\begin{enumerate}
\item $E_1,E_2,E_3,E_4$ the  exceptional divisors.
\item $H-E_4-E_2$, $H-E_4-E_3$,$H-E_4-E_1$ the strict transforms of the lines through two points.
\item $H-E_1-E_2-E_3$ the strict transform of the line through the three collinear points.
\end{enumerate}
The first two types of curves are $(-1)$- and the last one is a $(-2)$-curve. Analogously as in the previous section, we get the following result.
\begin{thm}
Let $D$ be a big divisor on $L_3$. Let furthermore $C=H-E_1-E_2-E_3$ be the line through the three chosen points, and $P\in C$ an arbitrary point. Then $\Delta_{C\supset \{P\}}(D)$ induces a normal toric degeneration.
\end{thm}
\qed

Since $H-E_1-E_2-E_3$ is the only $(-2)$-curves, we can use an analog of Lemma \ref{lemzarint} in order to compute Newton-Okounkov bodies.

\begin{ex}
Let us consider the divisor $D=4H-E_1-E_2-E_3-E_4$. This is an ample divisor.
We want to compute the Newton-Okounkov body with respect to the curve $C=H-E_1-E_2-E_3$ and a very general point $P$ on $C$.
\begin{figure}[H]
\begin{center}
\includegraphics[scale=0.5]{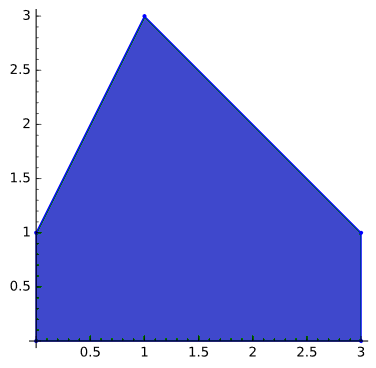}
\end{center}
\label{exok2}
\caption{N.-O. body of $D=4H-E_1-\dots- E_4$ on $L_3$}

\end{figure}
The Hilbert  polynomial of $D$ is given by:
\begin{align}
P_D(t)= 6t^2+4t+ 1.
\end{align}

\end{ex}

\section{Global Newton-Okounkov bodies on surfaces}
In this section we want to use our previous findings in order to compute global Newton-Okounkov bodies on (weak) del Pezzo surfaces. We will see that under good conditions the global semigroup $\Gamma_{Y_\bullet}(X)$ is finitely generated. Moreover, we will see how the generators of this semigroup give rise to generators of the Cox ring. We will illustrate our results for the varieties $X_5$ and $L_3$.

\subsection{ Generators of the global Newton-Okounkov body on surfaces}
In this section we want to generalize results from \cite{SS16} to arbitrary admissible flags.

Let us start by defining what we mean by a global Newton-Okounkov body.
\begin{defi}
Let $X$ be a projective variety.
Let $Y_\bullet$ be an admissible flag on $X$. Then we define the \emph{global Newton-Okounkov body of $X$} with respect to $Y_\bullet$ as the closure of
\begin{align}
\text{Cone}(\{ (\nu_{Y_\bullet}(s),[D]) \ | \ s\in H^0(X,\struc_X(D), \ D\in \text{Pic}(X))\}
\end{align}
in $\mathbb{R}^d\times N^1(X)_\mathbb{R}$. We denote it by $\Delta_{Y_\bullet}(X)$.
\end{defi}
Note that for any big divisor $D$, we have the following identity
\begin{align}
\Delta_{Y_\bullet}(D)=\Delta_{Y_\bullet}(X)\cap (\mathbb{R}^d\times \{[D]\}).
\end{align}

We will now focus on the case where $X$ is a smooth surface. Moreover, let us assume that $X$ admits a rational polyhedral pseudo-effective cone, e.g. if $X$ is a Mori dream surface. Let $C$ be a curve on $X$ and $P\in C$ a smooth point on $C$.
Let
\begin{align} 
D_1=P_1+N_1,\dots, D_r=P_r+N_r
\end{align}
 be the set of generators of the Zariski chambers with the property that $C$ is not contained in the support of the negative parts $N_1,\dots, N_r$.
The following is a  generalization of \cite[Theorem 3.2]{SS16} to arbitrary flags. Note that the proof works quite similar as the mentioned one.

\begin{thm}\label{thmglobalok}
Consider the notation introduced above.
The generators of the global Newton-Okounkov body $\Delta_{C\supset\{P\}}(X)$ are given by 
\begin{itemize}
\item  $(1,0,[C])$
\item 
$(0,\operatorname{ord}_{P}(N_{i|C}),[D_i]) \quad \text{for } i=1,\dots, r$
\item $(0,\operatorname{ord}_{P}(N_{i|C})+(P_i\cdot C),[D_i]) \quad \text{for } i=1,\dots, r$.
\end{itemize}
\end{thm}
\begin{proof}

It is not hard to see that all the above points are contained in $\Delta_{C\supset\{P\}}(X)$.

Let us now show that all points in $\Delta_{C\supset\{P\}}(X)$ are positive linear combinations of the above points.
It is enough to show that all valuation points are of this kind.
Let $D$ be a big divisor and $s\in H^0(X,\struc(D))$ an arbitrary section. Define $a:=\operatorname{ord}_{C}(s)$, consider $D^\prime:=D-aC$ and set $\xi:=s/s_{C}^a$ where $s_C$ is a defining section of $C$. We have 
\begin{align}
(\nu(s),[D])= a\cdot(1,0,[C])+(\nu(\xi),[D^\prime]).
\end{align}
Therefore, it is enough to show that $(\nu(\xi),[D^\prime])$ is a positive linear combination of the above points.
Let $D_{i_1},\dots, D_{i_s}$ be the generators of the unique Zariski chamber which contains the divisor $D^\prime$.
Then we can write
\begin{align}
D^\prime= \sum_{k=1}^s t_k\cdot D_{i_k}.
\end{align}
Furthermore, for the negative part $N^\prime:=N(D^\prime)=\sum  t_kN_{i_k}$ and $P^\prime:=P(D^\prime)=\sum t_k P_{i_k}$.
By definition of $D^\prime$,  we get that $C$ is not contained in the support of the negative part $N(D^\prime)$.
This also shows that $C$ is not contained in the negative parts of the $D_{i_k}$. 

Let us now choose $m\in \mathbb{Z}$ such that $mN^\prime$ and $mP^\prime$ are both integral.
We can decompose $\xi^m=\zeta \sigma$ for $\zeta \in H^0(X,\mathcal{O}_X(mP^\prime))$ and $\sigma\in H^0(X,\mathcal{O}_X(mN^\prime))$.
Then 
\begin{align}
m\cdot(\nu(\xi),[D^\prime])=(\nu(\zeta)+\nu(\sigma),[mP^\prime+mN^\prime])= (\nu(\zeta),[mP^\prime])+(\nu(\sigma),[mN^\prime]).
\end{align}
Furthermore,
\begin{align}
\nu(\sigma)=(0,\operatorname{ord}_{P}(\sigma_{|C}))=m\cdot\sum_{k=1}^s t_k\cdot (0,\operatorname{ord}_{P}(N_{i_k|C}).
\end{align}
On the other hand, we have
\begin{align}
\nu(\zeta)=(0,bm) \text{ where } b \in [0,P^\prime\cdot C]
\end{align}
Thus there is a $c\in [0,1]$ such that
\begin{align}
\nu(\zeta)= cm\cdot \sum_{k=1}^s t_k\cdot (0,0) + (1-c)m \sum_{k=1}^s t_k\cdot(0,P_k\cdot C).
\end{align}
Putting everything together we get
\begin{align}
m\cdot (\nu(\xi),[D])=& \left(  mc\cdot \sum_{k=1}^s t_k(0,\operatorname{ord}_{P}(N_{i_k|C})) + m(1-c)\bigg(\sum_{k=1}^s t_k(0,\operatorname{ord}_{P}(N_{i_k|C})\right.  \\
& \left. +P_k\cdot C\bigg),\left[ \sum_{k=1}^s t_k\cdot D_{i_k} \right] \right) =\\
=& mc \cdot \sum_{k=1}^s t_k \left(0,\operatorname{ord}_{P}(N_{i_k|C}),[D_{i_k}]\right)+\\
+& m(1-c)\cdot\sum_{k=1}^s t_k\left(0,\operatorname{ord}_{P}(N_{i_k|C}\right)+\left(P_{i_k}\cdot C\right)),[D_{i_k}]).
\end{align}
This proves the claim.
\end{proof}
The next proposition gives a more concrete characterization of the above mentioned divisors $D_i$.
\begin{prop}\label{propzargen}
Let $D$ be a divisor which spans an extremal ray of the closure of a Zariski chamber $\overline{\Sigma_P}$. Then $D$ spans an extremal ray of either the pseudo-effective cone or the nef cone of $X$.
\end{prop}
\begin{proof}

Let $P$ be a big and nef divisor. 
Then we define 
\begin{align}
\operatorname{Face}(P)&:=\bigcap_{C\in \operatorname{Null}(P)} C^{\perp}\cap \operatorname{Nef}(X)\\
V^{\geq 0}(\operatorname{Null}(P))&:=\text{Cone}(\operatorname{Null}(P)).
\end{align}
Then by \cite[Proposition 1.8]{BKS04}, we have
\begin{align}
\operatorname{Big}(X)\cap \overline{\Sigma_P}=\operatorname{Big}(X)\cap \operatorname{Face}(P)+ V^{\geq 0}(\operatorname{Null}(P)).
\end{align}
Hence, the extremal rays of $\overline{\Sigma_P}$ are either extremal rays of $\operatorname{Face}(P)$ or of $V^{\geq 0}(\operatorname{Null}(P))$.
However, since $\operatorname{Face}(P)$ is a face of the Nef cone, the first set of extremal rays lies inside the set of extremal rays of the Nef cone.
The extremal rays of  $V^{\geq 0}(\operatorname{Null}(P))$  are all negative, and thus extremal rays of the pseudo-effective cone.

\end{proof}

\begin{remark}
Proposition \ref{propzargen} combined with Theorem \ref{thmglobalok} is in some sense surprising. It shows that in order to calculate the global Newton-Okounkov body on a surface $X$, it is not necessary to know the exact structure of the Zariski chambers. It is not even necessary to compute any Zariski decomposition at all. However, in order to derive the structure of the generators of the global Newton-Okounkov body we heavily relied on the fact that Zariski decomposition as well as the Zariski chamber decomposition does exist.
\end{remark}
\subsection{Finite generation of the global semigroup}
We have seen above, that a smooth surface $X$ with a rational polyhedral pseudo-effective cone admits  rational polyhedral global Newton-Okounkov bodies with respect to all admissible flags. In this section we want to prove a stronger property, namely finite generation of the global semigroup appearing in the construction of global Newton-Okounkov bodies. We will prove this property  for the examples we have dealt with so far.
In order to prove such a statement, we need to consider Newton-Okounkov bodies of effective but not big divisors.
There are two different ways of defining these Newton-Okounkov bodies which both coincide for big divisors. One way is to define it via taking a fiber of the global Newton-Okounkov body. The corresponding body is called the \emph{numerical Newton-Okounkov body}. More concretely, we have
\begin{align}
\Delta_{Y_\bullet}^{num}(D):=\Delta_{Y_\bullet}(X)\cap \left( \mathbb{R}^d\times \{[D]\}\right).
\end{align}
Another way to associate a convex body to an effective divisor  is to just use the same definition as for big divisors. The resulting body is called the \emph{valuative Newton-Okounkov body}. More, concretely we define
\begin{align}
\Delta_{Y_\bullet}^{val}(D):= \overline{\text{Cone}(\Gamma_{Y_\bullet}(D))} \cap \left( \mathbb{R}^d\times\{1\}\right)
\end{align}
where 
\begin{align}
\Gamma_{Y_\bullet}(D):=\{(\nu_{Y_\bullet}(s),k) \ | \ k\in \mathbb{N}, s \in H^0(X,\mathcal{O}(kD))\setminus\{0\} \}.
\end{align}
In general, we have $\Delta_{Y_\bullet}^{val}(D)\neq\Delta_{Y_\bullet}^{num}(D)$. However, if $D$ is big the mentioned equality holds.

\begin{lem}\label{lemnumval} Let $X$ be a smooth Mori dream surface, $D$  an effective divisor on $X$ and $Y_\bullet$ an admissible flag such that $-K_X-Y_1$ is big and nef. Then $\Delta_{Y_\bullet}(D)^{num}=\Delta_{Y_\bullet}(D)^{val}$.
\end{lem}
\begin{proof}
Without loss of generality we might assume that $D$ is nef.
Following \cite{CPW17}, there are two different cases for $\Delta_{Y_\bullet}^{num}(D)$.
The first case is that $\mu:=\sup \{ t \ :  \ D-tY_1 \text{ is effective }\}$ is equal to $0$.
Then 
\begin{align}
\Delta^{num}_{Y_\bullet}(D)=\{(0,x) \ | \ x\in [0,D\cdot Y_1] \}.
\end{align}
Since $\mu=0$, we can deduce that
\begin{align}
\Delta_{Y_\bullet}^{val}(D)=\{0\}\times \Delta^{val}_{X|Y_1}(D)=\{0\}\times \Delta^{val}_{Y_1}(D).
\end{align}
Note that for the last identity we have used the fact that $H^1(X,D-Y_1)=0$ which means that the restriction morphism $H^0(X,\mathcal{O}_X(D))\to H^0(Y_1,\mathcal{O}_{Y_1}(D))$ is surjective.
However, it easily follows that $\Delta_{Y_1}^{val}(D_{|Y_1})=[0,D\cdot Y_1]$. This proves $\Delta_{Y_\bullet}^{val}(D)=\Delta_{Y_\bullet}^{num}(D)$ in for the case $\mu=0$.

Suppose now that $\mu>0$. Then $\Delta_{Y_\bullet}^{num}$ is given by a line segment $\text{Conv}\{(0,0),(\mu,Q)\}$ for some number $Q\in\mathbb{R}_{\geq 0}$. Since $\Delta_{Y_\bullet}^{val}(D)\subseteq\Delta_{Y_\bullet}^{num}(D)$ it is enough to prove that there are sections $s_1,s_2\in H^0(X,\mathcal{O}(kD)$ such that $\nu_{Y_\bullet}(s_1)=(0,0)$, and $\nu_{Y_\bullet}(s_2)=k(\mu,Q)$.
However, since $D$ is nef and thus semi ample, the first assertion is clear. Moreover, since $D-\mu C$ is effective, the second assertion follows.
This proves the claim.
\end{proof}

\begin{lem}\label{lemsemnorm}
Let $X$ be a smooth del Pezzo surface. Let $Y_\bullet$ be an admissible flag such that $-K_X-Y_1$ is big and nef, and let $Y_1$ be rational, i.e. of genus 0. Then for all effective divisors $D$, the semigroup $\Gamma_{Y_\bullet}(D)$ is finitely generated normal.
\end{lem}
\begin{proof}
This proof works similarly as the proof of Theorem \ref{thmnormaldegminusone}.
\end{proof}

\begin{remark}\label{remsemnorm}
The above lemma is also valid for the varieties $L_3$ and $S_6$, if we take as $Y_1$ the single $(-2)$-curve.
\end{remark}

\begin{thm}
Suppose one of the following situations.
\begin{itemize}
\item $X=X_r$ is the blow-up of $1\leq r\leq 6$ points in general position and $Y_\bullet$ is an admissible flag such that $Y_1$ is negative. 
\item $X=L_3$ or $X=S_6$ and  $Y_\bullet$ is an admissible flag such that $Y_1$ is the corresponding single $(-2)$-curve.
\end{itemize}
Then the global semigroup 
\begin{align}
\Gamma_{Y_\bullet}(X)=\{(\nu_{Y_\bullet}(s),D) \ | \ D\in N^1(X)=\text{Pic}(X), \ s\in H^0(X,\mathcal{O}(D)) \}
\end{align}
 is finitely generated normal.
\end{thm}
\begin{proof}
We know, by Theorem \ref{thmglobalok}, that $\Delta_{Y_\bullet}(X)=\overline{\text{Cone}(\Gamma_{Y_\bullet}(X))}$ is rational polyhedral.
We want to prove that $\overline{\text{Cone}(\Gamma_{Y_\bullet}(X))}\cap (\mathbb{Z}^2\times N^1(X))=\Gamma_{Y_\bullet}(X)$. Then the result follows from Gordan's lemma.
Consider $(a,D)\in \overline{\text{Cone}(\Gamma_{Y_\bullet}(X))}$ for $a\in \mathbb{Z}^2$ and $D$ an integral effective divisor in $N^1(X)$.
This means that 
\begin{align}
a\in \Delta^{num}_{Y_\bullet}(D)=\Delta^{val}_{Y_\bullet}(D)=\overline{\text{Cone}(\Gamma_{Y_\bullet}(D))}\cap \left(\mathbb{R}^2\times\{1\}\right).
\end{align}
But by Lemma \ref{lemsemnorm} and Remark \ref{remsemnorm}, $\Gamma_{Y_\bullet}(D)$ is normal. Thus, there is a section $s\in H^0(X,\mathcal{O}(D))$ such that $\nu_{Y_\bullet}(s)=a$. This proves that $(a,D)\in \Gamma_{Y_\bullet}(X)$.

\end{proof}

The finite generation of the global semigroup $\Gamma_{Y_\bullet}(X)$ has the following consequences for the Cox ring $\text{Cox}(X)$.

\begin{thm}\label{thmcox}
Let $X$ be a  $\mathbb{Q}$-factorial variety with $N^1(X)=\operatorname{Pic}(X)$. Let $Y_\bullet$ be an admissible flag.
Suppose $\Gamma_{Y_\bullet}(X)$ is finitely generated by 
\begin{align*}
(\nu_{Y_\bullet}(s_1),D_1),\dots (\nu_{Y_\bullet}(s_N),D_n).
\end{align*}
 Then the Cox ring $\operatorname{Cox}(X)$ is generated by the sections $s_1,\dots, s_N$.
\end{thm}
\begin{proof}
Let $R$ be the $\mathbb{C}$-algebra which is generated by the sections $s_1,\dots,s_N$.
Let $D$ be any effective divisor in $X$. Let $k=h^0(X,\mathcal{O}_X(D))=\vert \nu_{Y_\bullet}(H^0(X,\mathcal{O}_X(D))\setminus\{0\})\vert$.
Since the $(\nu_{Y_\bullet}(s_1),D_1),\dots (\nu_{Y_\bullet}(s_N),D_n)$ generate $\Gamma_{Y_\bullet}(X)$, it follows that there are $f_1,\dots,f_k\in R\cap H^0(X,\mathcal{O}_X(D))\setminus\{0\}$ which all have a different valuation. But then it follows from \cite[Proposition 2.3]{KK12}
that $f_1,\dots, f_k$ are linearly independent. This proves that they form a basis of $H^0(X,\mathcal{O}_X(D))$ and that every section $s\in H^0(X,\mathcal{O}(D))$ lies in the algebra $R$. This show that $R\cong \operatorname{Cox}(D)$.
\end{proof}

\subsection{Examples of global Newton-Okounkov bodies and global semigroups}
In this last paragraph we want to consider  two concrete examples and compute their global Newton-Okounkov bodies. In the second example we also present generators of the global semigroup and use them to find generators of the Cox ring.

\begin{ex} First of all we consider the del Pezzo surface $X_5$, which is the blow-up of five points in general position in $\mathbb{P}^2$.
As a flag, we take the negative curve $C:=H-E_1-E_2$, and a general point on it.
According to Theorem \ref{thmglobalok}, we need to compute all ray generators of Zariski chambers, whose support of the negative part does not contain the negative curve $C$. Using Proposition \ref{propzargen}, these are given by all the negative curves except the curve $C$ and the generators of the extremal rays of the nef cone.

With the help of a computer calculation, we compute the global Newton Okounkov body and present the  resulting hyperplane representation in $\mathbb{R}^2\times N^1(X_5)_\mathbb{R}\cong \mathbb{R}^8$. Choosing $H,E_1,\dots,E_5$ as a basis for $N^1(X_5)_\mathbb{R}$ we get the following representation for $\Delta_{Y_\bullet}(X_5)$:
\begin{align}
\tiny
\begin{pmatrix}
1& 0& 0& 0& 0& 0& 0& 0 \\  1& -1& 1& 1& 1& 0& 0& 0 \\  0& -1& 2& 1& 1& 0& 1& 1 \\  -1& -1& 3& 1& 1& 1& 1& 2 \\  0& -1& 2& 1& 1& 1& 0& 1 \\  -1& -1& 2& 0& 1& 1& 1& 0 \\  0& -1& 1& 0& 1& 0& 0& 0 \\  -1& -1& 3& 1& 1& 2& 1& 1 \\  -1& -1& 3& 1& 1& 1& 2& 1 \\  0& -1& 2& 1& 1& 1& 1& 0 \\  -1& -1& 2& 0& 1& 0& 1& 1 \\  -1& -1& 2& 0& 1& 1& 0& 1 \\  0& -1& 1& 1& 0& 0& 0& 0 \\  -1& -1& 2& 1& 0& 1& 1& 0 \\  -1& -1& 2& 1& 0& 1& 0& 1 \\  -2& -1& 2& 0& 0& 1& 0& 1 \\  -2& -1& 2& 0& 0& 1& 1& 0 \\  -1& -1& 1& 0& 0& 0& 0& 0 \\  -1& -1& 2& 1& 0& 0& 1& 1 \\  -2& -1& 2& 0& 0& 0& 1& 1 \\  -2& -1& 3& 1& 0& 1& 1& 2 \\  -2& -1& 3& 1& 0& 1& 2& 1 \\  -2& -1& 3& 1& 0& 2& 1& 1 \\  -2& -1& 3& 0& 1& 1& 1& 2 \\  -2& -1& 3& 0& 1& 1& 2& 1 \\  -2& -1& 3& 0& 1& 2& 1& 1 \\  -2& -1& 4& 1& 1& 2& 2& 2 \\  -1& 0& 2& 1& 0& 1& 1& 1 \\  0& 1& 0& 0& 0& 0& 0& 0 \\  -1& 0& 1& 0& 0& 1& 0& 0 \\  -1& 0& 1& 0& 0& 0& 0& 1 \\  -1& 0& 1& 0& 0& 0& 1& 0 \\  -1& 0& 2& 0& 1& 1& 1& 1 \\  -3& -1& 4& 0& 1& 2& 2& 2 \\  -3& -1& 4& 1& 0& 2& 2& 2 \\  -4& -1& 4& 0& 0& 2& 2& 2 \\  -3& -1& 3& 0& 0& 2& 1& 1 \\  -3& -1& 3& 0& 0& 1& 1& 2 \\  -3& -1& 3& 0& 0& 1& 2& 1
\end{pmatrix}\cdot (x_1,\dots, x_8)^T\leq 0.
\end{align} 
It is a convex cone in $\mathbb{R}^8$ which is defined by a minimal number of $39$ inequalities or a minimal number of $22$ rays.
Note that the above equations give an Ehrhart type formula for the Hilbert polynomial of a given divisor $D=(x_3,\dots,x_8)$ similar to the one  derived in 
\cite[Example 1.3]{SX10}. 

\end{ex}
\begin{ex}
Consider now $L_3$, which is the blow-up of four points such that three of them lie on a line. Let us suppose that $P_1,\dots, P_3$ lie on a line. We choose $H,E_1,\dots, E_4$ as our basis for $N^1(L_3)_\mathbb{R}$.
Then the global Newton-Okounkov body of $L_3$ with respect to the curve $C=H-E_1-E_2-E_3$ and a general point on it is given by the following linear inequalities:
\begin{align}
\begin{pmatrix}
1& 0& 0& 0& 0& 0& 0\\2& -1& 1& 1& 1& 1& 0\\1& -1& 1& 0& 1& 1& 0\\1& -1& 1& 1& 0& 1& 0\\1& -1& 1& 1& 1& 0& 0\\-1& -1& 1& 0& 0& 0& 0\\0& -1& 1& 0& 0& 1& 0\\0& -1& 1& 0& 1& 0& 0\\0& -1& 1& 1& 0& 0& 0\\0& 1& 0& 0& 0& 0& 0\\-1& 0& 1& 0& 0& 0& 1
\end{pmatrix}\cdot (x_1,\dots,x_7)^T \geq 0.
\end{align}
The ray generators of the global Newton-Okounkov body are
\begin{align}
&(0, 0, E_4),\ (0, 0, E_3), \ (0, 0, E_2),\ (0, 0, E_1),\ (0, 0, H-E_1-E_4),\\ 
& ( 0, 0, H-E_2-E_4),\
 (0, 0, H-E_3-E_4),\ (0, 1, H-E_4),\\
 & ( 1, 0, H-E_1-E_2-E_3).
\end{align}
A calculation shows that these generators, form a Hilbert basis, so that they are a generating set of the global semigroup $\Gamma_{Y_\bullet}(L_3)$.
It follows from Theorem \ref{thmcox}, that $\text{Cox}(L_3)$ is generated by the following sections:
\begin{itemize}
\item the negative curves
\item the strict transform of a general line going through $P_4$.
\end{itemize}

\end{ex}

\end{document}